\documentclass{article}
\usepackage{graphicx}
\usepackage[left=3cm, top=2cm, right=3cm, bottom=3cm]{geometry}
\usepackage{listings}
\usepackage{float}
\usepackage{amssymb,amsmath,mathtools}
\usepackage{enumerate}
\usepackage[numbers,square,sort]{natbib}
\usepackage{wrapfig}
\usepackage{framed}
\usepackage{tabstackengine}
\usepackage{braket}
\usepackage[nottoc]{tocbibind}\settocbibname{References}
\usepackage{mathrsfs}

\usepackage{indentfirst}
\usepackage{setspace}
\usepackage{bm}
\usepackage[dvipsnames]{xcolor}

\usepackage{tikz}
\tikzset{every loop/.style={}}
\usetikzlibrary{arrows,shapes,decorations.markings,automata,backgrounds,petri,bending,calc}

\usepackage{fancyhdr}
\fancypagestyle{firststyle}
{
	
	\fancyfoot[l]{\footnotesize\textit{Date:} \today}
	\fancyfoot[c]{1}

}

\usepackage{tikz-cd} 
\tikzset{
    labl/.style={anchor=south, rotate=90, inner sep=.5mm}
}

\usepackage{enumitem}
\setlist[enumerate,1]{label=(\arabic*)}
\newlist{steplist}{enumerate}{1}
\setlist[steplist]{label={Step \arabic*:}, ref={Step \arabic*}}

\usepackage[utf8]{inputenc}
\usepackage[english]{babel} 

\usepackage{amsthm}
\newtheorem{thm}{Theorem}[section]
\newtheorem{lem}[thm]{Lemma}
\newtheorem{prop}[thm]{Proposition}
\newtheorem{cor}[thm]{Corollary}
\numberwithin{equation}{section}
\newtheorem{con}[thm]{Conjecture}

\theoremstyle{definition}
\newtheorem{defn}[thm]{Definition} 
\newtheorem{remk}[thm]{Remark}
\newtheorem{nota}[thm]{Notation}
\newtheorem{cons}[thm]{Construction}

\newenvironment{claim}[1]{\par\noindent\underline{Claim:}\space#1}{}
\newenvironment{claimproof}[1]{\par\noindent\underline{Proof:}\space#1}{\leavevmode\unskip\penalty9999 \hbox{}\nobreak\hfill\quad\hbox{$\blacksquare$}}

\newcommand{\sfC}{\textsf{C}}
\newcommand{\link}{\operatorname{link}}

\newcommand{\LCM}{\operatorname{LCM}}
\newcommand{\calP}{\mathcal{P}}

\newcommand{\calI}{\mathcal{I}}

\newcommand{\WProj}{\operatorname{WProj}}
\newcommand{\calB}{\mathcal{B}}

\newcommand{\wt}{\widetilde}

\newcommand{\N}{\mathbb{N}}

\newcommand{\Z}{\mathbb{Z}}

\newcommand{\acts}{\curvearrowright}

\makeatletter
\newsavebox{\@brx}
\newcommand{\llangle}[1][]{\savebox{\@brx}{\(\m@th{#1\langle}\)}%
	\mathopen{\copy\@brx\kern-0.5\wd\@brx\usebox{\@brx}}}
\newcommand{\rrangle}[1][]{\savebox{\@brx}{\(\m@th{#1\rangle}\)}%
	\mathclose{\copy\@brx\kern-0.5\wd\@brx\usebox{\@brx}}}
\makeatother

\begin{document}
\begin{center}
{\LARGE\bf
Imitator homomorphisms for special cube complexes}\\
\bigskip
\bigskip
{\large Sam Shepherd}
\end{center}
\bigskip

\begin{abstract}
Central to the theory of special cube complexes is Haglund and Wise's construction of the canonical completion and retraction, which enables one to build finite covers of special cube complexes in a highly controlled manner. In this paper we give a new interpretation of this construction using what we call imitator homomorphisms. This provides fresh insight into the construction and enables us to prove various new results about finite covers of special cube complexes -- most of which generalise existing theorems of Haglund--Wise to the non-hyperbolic setting. In particular, we prove a convex version of omnipotence for virtually special cubulated groups.
\end{abstract}
\bigskip
\tableofcontents

\bigskip
\section{Introduction}

Gromov's definition of CAT(0) cube complex \cite{Gromov87} laid the foundation for an active subfield of Geometric Group Theory, driven by Wise, Sageev and others, in which groups are studied via their actions on such cube complexes. 
Many groups have been cubulated (sometimes referred to as cocompactly cubulated), including small cancellation groups, finite-volume hyperbolic 3-manifold groups, limit groups, many Coxeter groups, and hyperbolic free-by-cyclic groups -- see for example \cite{WiseRiches}.
Consequences of cubulating a group include that the group is CAT(0), bi-automatic \cite{Swiatkowski06} and that it satisfies the Tits Alternative \cite{SageevWise05}. Haglund and Wise introduced an important subclass of cube complexes, called special cube complexes, which grant groups even stronger properties \cite{HaglundWise08}. Any group which is virtually the fundamental group of a finite special cube complex, referred to as a \emph{virtually special group}, is residually finite, $\Z$-linear, conjugacy separable and has many separable subgroups \cite{HaglundWise08,Minasyan12}. Agol proved that hyperbolic cubulated groups are virtually special, and used this to resolve the long-standing Virtual Haken Conjecture in 3-manifold theory \cite{Agol13}. It also showed that many groups, including small cancellation groups and closed hyperbolic 3-manifold groups, satisfy the strong properties granted by special cube complexes.
In fact, being hyperbolic and cubulated is equivalent to being hyperbolic and virtually special \cite{BergeronWise12}.
 Many of the greatest achievements in the theory of virtually special groups have thus far been restricted to the (relatively) hyperbolic setting, for example Wise's Quasi Convex Hierarchy Theorem and Malnormal Special Quotient Theorem \cite{WiseQCH,WiseRiches,AgolGrovesManning16}.

At the heart of Haglund and Wise's theory of special cube complexes is the construction of the canonical completion and retraction for a local isometry $\phi:Y\to X$ of (directly) special cube complexes \cite{HaglundWise08}, which enables one to build finite covers of special cube complexes in a highly controlled manner (see Section \ref{subsec:directlyspecial} for the difference between special and directly special cube complexes). In this paper we give a new interpretation of this construction: we imagine two people wandering around in these cube complexes, the \emph{walker} wanders around $X$ while the \emph{imitator} wanders around $Y$, and the imitator tries to copy the walker by crossing over the same hyperplanes (pulled back along $\phi$).
The \emph{imitator homomorphism of $\phi$}, which maps from a certain subgroup of $\pi_1 X$ into $\pi_1 Y$, is defined by inputting the movements of the walker and outputting the movements of the imitator.
The precise descriptions are given in Section \ref{sec:hom}. Our interpretation not only provides greater insight into the Haglund--Wise construction, but also facilitates new applications. We use these constructions to prove various results about virtually special cube complexes and virtually special groups, most of which generalise existing results of Haglund and Wise to the non-hyperbolic setting. We now describe our main results. First we need a definition.

\begin{defn}(Commanding group elements)\\\label{defn:commandingelts}
	A group $G$ \emph{commands} a set of elements $\{g_1,...,g_n\}\subset G$ if there exists an integer $N>0$ such that for any integers $r_1,...,r_n>0$ there exists a homomorphism to a finite group $G\to\bar{G},g\mapsto\bar{g}$ such that the order of $\bar{g}_i$ is $Nr_i$. If this can always be done with $\langle\bar{g}_i\rangle\cap\langle\bar{g}_j\rangle=\{1\}$ for all $i\neq j$ then we say that $G$ \emph{strongly commands} $\{g_1,...,g_n\}$.
\end{defn}

Clearly a group can't command arbitrary sets of elements, for instance conjugate elements must have the same order in finite quotients. The best one can hope for is that a group commands any \emph{independent} set of elements $\{g_1,...,g_n\}$ (meaning the $g_i$ have infinite order and no non-zero power of $g_i$ is conjugate to a non-zero power of $g_j$ for $i\neq j$) -- such a group is called \emph{omnipotent}.
Omnipotence was defined by Wise in \cite{Wise00}, where he proved that free groups are omnipotent. Subsequently omnipotence has been proven for surface groups \cite{Bajpai07}, Fuchsian groups \cite{Wilton10}, and virtually special hyperbolic groups \cite{WiseQCH}. This third example naturally leads to the question of whether all virtually special groups are omnipotent. This however is far from true, for example $\Z^2$ does not command the set of elements $\{(1,0),(0,1),(1,1)\}$ (note this set is independent but not linearly independent). Our big insight is that virtually special groups do command independent sets of elements if we add the additional assumption that the elements are \emph{convex} with respect to the given cubulation (see Definition \ref{defn:convexsubgroups}). For the standard cubulation of $\Z^2$ by a square tiling of the plane, $(1,0)$ and $(0,1)$ are convex but not $(1,1)$.
In general, convexity of a group element depends on the choice of cubulation, however it follows from the Cubical Flat Torus Theorem \cite{WiseWoodhouse17} that an infinite order element will be convex with respect to every cubulation if it has no non-trivial power contained in a $\Z^2$ subgroup.
The following theorem is proved in Section \ref{subsec:thmcommand1}. See Section \ref{subsec:cubulated} for the precise definitions involved. As an application, we use it to deduce that right-angled Artin groups command random sets of elements (Theorem \ref{thm:RAAG}).

\begin{thm}\label{thm:command1}
	Every virtually special cubulated group $G\acts X$ strongly commands every independent set of convex elements.
\end{thm}

We say that a group is \emph{strongly omnipotent} if it strongly commands any independent set of elements. This notion was considered by Bridson and Wilton in \cite{BridsonWilton15}, where they proved that virtually free groups are strongly omnipotent -- this was a key step in their proof of the undecidability of the triviality problem for profinite completions. As a corollary of Theorem \ref{thm:command1}, we obtain strong omnipotence for virtually special hyperbolic groups (the hyperbolicity ensures that all infinite order elements are convex with respect to any cubulation \cite[Proposition 7.2]{HaglundWise08}), thus generalising the theorems of Wise and Bridson--Wilton. Note that the proof of Theorem \ref{thm:command1} does not rely on any previous omnipotence results, so can also be viewed as a new proof for these results.

\begin{cor}
All virtually special hyperbolic groups are strongly omnipotent.
\end{cor}

The notion of commanding a collection of elements extends naturally to a notion of commanding a collection of subgroups (Definition \ref{defn:commandingsubgroups}). This provides a powerful tool for constructing finite covers of graphs of groups if the vertex groups command their incident edge groups (Proposition \ref{prop:GoG}).
Although the terminology of commanding is new to this paper, there are similar ideas in the literature which have been used to study graphs of groups, which we discuss in Section \ref{subsec:GoG}.
In particular, we have the following deep theorem as a consequence of Wise's Malnormal Special Quotient Theorem.
This theorem was explained to the author by Woodhouse before the author formulated the notion of commanding.
Note that a subgroup of a hyperbolic cubulated group $G\acts X$ is convex if and only if it is quasiconvex \cite[Proposition 7.2]{HaglundWise08}.

\begin{thm}\label{thm:commandqc}
Every virtually special hyperbolic group commands every almost malnormal collection of quasiconvex subgroups.
\end{thm}

We prove this result in Section \ref{subsec:commanddefn}, where we also extend it to the relatively hyperbolic setting (Theorem \ref{thm:commandrel}).
We conjecture that this extends to the non-hyperbolic setting as well.

\begin{con}\label{con:command1}
Every virtually special cubulated group $G\acts X$ commands every almost malnormal collection of convex subgroups.
\end{con}

We fall short of proving Conjecture \ref{con:command1}, although we do prove some weaker versions (Theorems \ref{thm:nonregc} and \ref{thm:weakcommandqc}). We also prove the following theorem in Section \ref{subsec:control}, which takes us most of the way towards proving Theorem \ref{thm:command1}. We note that the homomorphisms $\rho_i$ are restrictions of imitator homomorphisms.

\begin{thm}\label{thm:rhoi}(Independent imitator homomorphisms)\\
	Let $G\acts X$ be a virtually special cubulated group, and let $P_1,...,P_n<G$ be convex subgroups, such that all conjugates of $P_i$ and $P_j$ have finite intersection if $i\neq j$. Then there is a finite-index normal subgroup $\dot{G}\triangleleft G$, and for each $i$ there is a homomorphism $\rho_i:\dot{G}\to P_i$ such that:
	\begin{enumerate}
		\item $\rho_i$ is the identity on $\dot{P}_i:=P_i\cap\dot{G}$.
		\item $\rho_i(gP_j g^{-1}\cap\dot{G})=\{1\}$ for $g\in G$ and $i\neq j$.
	\end{enumerate}
\end{thm}

Other than imitator homomorphisms, the main ingredient in the proof of Theorem \ref{thm:rhoi} is the following theorem.

\begin{thm}(Elevating to trivial wall projections)\\\label{thm:trivialwallproj1}
	Let $Y_1,Y_2\to X$ be local isometries of finite virtually special cube complexes, and let $K_1,K_2<\pi_1 X$ be the corresponding subgroups (well-defined up to conjugacy). Suppose that $K_1$ has trivial intersection with every conjugate of $K_2$. Then there is a finite directly special cover $\hat{X}\to X$ such that all elevations of $Y_1$ and $Y_2$ are embedded, and each elevation of $Y_1$ has trivial wall projection onto each elevation of $Y_2$.
\end{thm}

This partially generalises a theorem of Haglund and Wise \cite[Corollary 5.8]{HaglundWise12} to the non-hyperbolic setting. Trivial wall projection is defined in Definition \ref{defn:wallprojection}; it is a concept due to Haglund and Wise that has a natural connection to imitator covers. Roughly speaking, the utility of Theorem \ref{thm:trivialwallproj1} is that it allows us to control the elevations of $Y_1$ and $Y_2$ independently when passing to further finite covers of $\hat{X}$. This was a key ingredient in the proof of Haglund and Wise's combination theorem for special cube complexes \cite{HaglundWise12}. The main reason Theorem \ref{thm:trivialwallproj1} is only a partial generalisation of \cite[Corollary 5.8]{HaglundWise12} is because it does not say anything about wall projections between different elevations of $Y_1$ (or $Y_2$); we believe this is the main hurdle to proving Conjecture \ref{con:command1}, and also to potentially proving a non-hyperbolic combination theorem for special cube complexes.

The proof of Theorem \ref{thm:trivialwallproj1} uses the separability of triple cosets of convex subgroups in virtually special cubulated groups -- another new result of this paper. The separability of convex subgroups themselves is due to \cite[Corollary 7.9]{HaglundWise08}, while the separability of double cosets of convex subgroups is due to \cite[Theorem A.1]{Oregonreyes20}. We give proofs of all three statements in Section \ref{sec:applysingle}, which all make direct use of imitator homomorphisms.

A key criterion used throughout the paper is the exclusion of inter-osculations between a subcomplex and a hyperplane (Definition \ref{defn:complexosculate}). These subcomplex-hyperplane inter-osculations first appeared in \cite[Remark A.9]{Oregonreyes20}, and they generalise the usual hyperplane-hyperplane inter-osculations, which is one of the pathologies excluded in the definition of special cube complex.

In Section \ref{sec:hierarchical} we construct covers using hierarchies of imitators, and prove the following theorem. This generalises \cite[Theorem 4.25]{HaglundWise12} to the non-hyperbolic setting. 

\begin{thm}(Virtually connected intersections)\\\label{thm:connected1}
	For $i=1,...,n$ let $(Y_i,y_i)\to(X,x)$ be based local isometries of finite virtually special cube complexes. Then there is a finite cover $(\dot{X},\dot{x})\to(X,x)$ such that the based elevations $\dot{Y}_i$ of $Y_i$ are embedded in $\dot{X}$ and have connected intersections $\cap_{i\in E}\dot{Y}_i$ for any $\emptyset\neq E\subset\{1,...,n\}$. Moreover, if the $Y_i$ are embedded in $X$ and do not inter-osculate with hyperplanes of $X$, then we may assume that the covering map $\dot{X}\to X$ is injective on $\cap_{i=1}^n \dot{Y}_i$. 
\end{thm}

It is also worth noting that one can obtain explicit bounds for all of our theorems by tracing through the proofs. For example, in Theorem \ref{thm:command1}, if $G$ is the fundamental group of a finite directly special cube complex $X$, and if the set of convex elements correspond to local isometries $Y_1,...,Y_n\to X$, then the integer $N$ from Definition \ref{defn:commandingelts} can be bounded by an explicit formula in terms of the sizes of $X$ and the $Y_i$. Similarly, the size of the finite group $\bar{G}$ in Definition \ref{defn:commandingelts} can be bounded by an explicit formula in terms of the sizes of $X$ and the $Y_i$ and the integers $r_1,...,r_n$ from Definition \ref{defn:commandingelts}. We do not actually compute any such formulae in this paper however.

\textbf{Acknowledgements:}\,
Thanks to Martin Bridson, Dawid Kielak, Mark Hagen, Ashot Minasyan and the anonymous referee for their comments and corrections. I am also grateful for discussions with Daniel Woodhouse, where he explained how Theorems \ref{thm:commandqc}, \ref{thm:commandrel} and \ref{thm:hypRF} can be deduced from the literature.

\bigskip
\section{Preliminaries}\label{sec:prelim}

See \cite{HaglundWise08} for the definitions of cube complex, hyperplane and non-positively curved cube complex. In this section we recall some more specialised definitions regarding cube complexes and specialness. First some notational conventions.

\begin{nota}
	Let $X$ be a cube complex. We let $X^n$ denote its $n$-skeleton. We refer to the 0-cubes, 1-cubes and 2-cubes of $X$ as \emph{vertices}, \emph{edges} and \emph{squares} respectively.
	By a \emph{path} in $X$ we will always mean an edge path, and a \emph{loop} is an edge path that starts and finishes at the same vertex. We let $\link(x)$ denote the \emph{link} of a vertex $x$; this combinatorial complex is the space of directions at $x$, and is a simplicial complex if $X$ is non-positively curved. In this paper all cube complexes and their covers will be connected unless otherwise stated.
\end{nota}

\subsection{Directly special cube complexes}\label{subsec:directlyspecial}

\begin{defn}(Parallelism)\\
	Let $X$ be a cube complex. Two edges $e_1,e_2\in X^1$ are \emph{elementary parallel} if they appear as opposite edges of some square of $X$. The relation of elementary parallelism generates the equivalence relation of \emph{parallelism}. We write $e_1\parallel e_2$ if edges $e_1$ and $e_2$ are parallel. We define $H(e_1)$ to be the hyperplane dual to an edge $e_1$ -- note that $e_1\parallel e_2$ is equivalent to $H(e_1)=H(e_2)$.
\end{defn}

\begin{defn}(Intersecting and osculating hyperplanes)\\
	Let $X$ be a cube complex. Suppose distinct edges $e_1$ and $e_2$ of $X$ are incident at a vertex $x$.
	\begin{enumerate}
		\item If $e_1$ and $e_2$ form the corner of a square at $x$, then we say that the hyperplanes $H(e_1)$ and $H(e_2)$ \emph{intersect at $(x;e_1,e_2)$}. If in addition $H(e_1)=H(e_2)$, then we say that $H(e_1)$ \emph{self-intersects at $(x;e_1,e_2)$}. Note that a pair of hyperplanes intersect (as subsets of $X$) if and only if they are equal or they intersect at some $(x;e_1,e_2)$.
		\item If $e_1$ and $e_2$ do not form the corner of a square at $x$, then we say that the hyperplanes $H(e_1)$ and $H(e_2)$ \emph{osculate at $(x;e_1,e_2)$}. If in addition $H(e_1)=H(e_2)$, then we say that $H(e_1)$ \emph{self-osculates at $(x;e_1,e_2)$}. Alternatively, if $e_1$ has both its ends incident at $x$, then we say that $H(e_1)$ \emph{self-osculates at $(x;e_1)$}. We say that a pair of hyperplanes \emph{osculate} if they osculate at some $(x;e_1,e_2)$. 
		\item We say that distinct hyperplanes $H_1$ and $H_2$ \emph{inter-osculate} if they both intersect and osculate. 
	\end{enumerate} 
We will sometimes just say that a pair of hyperplanes intersect or osculate at a vertex $x$ if we do not wish to specify edges $e_1$ and $e_2$.
The notation from \cite{HaglundWise08} is slightly different from ours as they work with oriented edges and distinguish between direct and indirect self-osculations.
\end{defn}

\begin{defn}(Two-sided hyperplanes)\\\label{defn:twosided}
	A hyperplane $H$ is \emph{two-sided} if the map $H\to X$ extends to a combinatorial map $H\times[-1,1]\to X$ (where we consider $H$ with its induced cube complex structure).
\end{defn}

\begin{defn}(Directly special and virtually special cube complexes)\\
We say that a non-positively curved cube complex $X$ is \emph{directly special} if every hyperplane is two-sided, no hyperplane self-intersects or self-osculates, and no pair of hyperplanes inter-osculate. A non-positively curved cube complex $X$ is \emph{virtually special} if it has a finite-sheeted directly special cover.
\end{defn}

\begin{remk}
	The definition of directly special cube complex appears in \cite{Huang18}; it is slightly stronger than the original notion of specialness \cite{HaglundWise08} as it excludes all types of hyperplane self-osculation (including the existence of edge loops) and it requires hyperplanes to be two-sided. However, among finite non-positively curved cube complexes the notions of special and directly special are equivalent up to finite covers \cite[Proposition 3.10]{HaglundWise08}. Additionally, the cubical subdivision of a special cube complex is directly special.
\end{remk}

\subsection{Local isometries}

\begin{defn}(Local isometries)\\
	A \emph{local isometry} $\phi:Y\to X$ of non-positively curved cube complexes is a combinatorial map such that each induced map $\link(y)\to\link(\phi(y))$ is an embedding with image a full subcomplex of $\link(\phi(y))$ (a subcomplex $C$ of a simplicial complex $D$ is \emph{full} if any simplex of $D$ whose vertices are in $C$ is in fact entirely contained in $C$). A subcomplex $Y\subset X$ is \emph{locally convex} if the inclusion $Y\xhookrightarrow{}X$ is a local isometry.	
	If $X$ is also CAT(0) then we simply say that $Y$ is \emph{convex} (and this is equivalent to $Y$ being convex in the combinatorial or CAT(0) metrics on $X$).
\end{defn}

\begin{remk}
	Suppose $\phi:Y\to X$ is a local isometry of non-positively curved cube complexes which contain no edges that are loops (such as directly special cube complexes), and suppose $f_1,f_2$ are edges incident at $y\in Y^0$. Then $f_1,f_2$ form the corner of a square at $y$ if and only if $\phi(f_1),\phi(f_2)$ form the corner of a square at $\phi(y)$ (this is true without the no-edge-loop assumption if we work with oriented edges).
\end{remk}
\begin{remk}	
	If $\phi:Y\to X$ is a local isometry and $X$ is directly special, then $Y$ is directly special. Covering maps are local isometries, so in particular any cover of a directly special cube complex is directly special.		
\end{remk}

\begin{nota}\label{nota:link}
	If $X$ is a directly special cube complex and $x\in X^0$, then we write $e\in\link(x)$ if $e$ is an edge incident at $x$. This notation makes sense because only one end of $e$ is incident at $x$ ($H(e)$ doesn't self-osculate), and so $e$ defines a unique vertex of $\link(x)$.
\end{nota}

\subsection{More intersections and osculations}

The notion of two hyperplanes in a cube complex $X$ intersecting or osculating generalises to the notion of a hyperplane and a complex $Y\to X$ intersecting or osculating as follows. This generalisation first appeared in \cite[Remark A.9]{Oregonreyes20}.

\begin{defn}(Intersections and osculations of hyperplanes with complexes)\\\label{defn:complexosculate}
	Let $\phi:Y\to X$ be a local isometry of non-positively curved cube complexes. If $H$ is a hyperplane of $Y$, we denote by $\phi[H]$ the unique hyperplane of $X$ that contains the image of $H$ (equivalently, if $H=H(f)$ then $\phi[H]=H(\phi(f))$). Now let $H$ be a hyperplane in $X$.
	\begin{enumerate}
		\item We say that $H$ and $Y$ \emph{intersect} if $H$ intersects $\phi(Y)$, or equivalently if $H=\phi[H']$ for some hyperplane $H'$ in $Y$.
		\item We say that $H$ and $Y$ \emph{osculate at $(y;e)$} if $y\in Y^0$, $e\in X^1$ is an edge incident to $\phi(y)$, $H=H(e)$, and no edge $f\in Y^1$ incident to $y$ has $\phi(f)=e$. 
		\item We say that $H$ and $Y$ \emph{inter-osculate} if they both intersect and osculate.
	\end{enumerate}
	If $Y\subset X$ is a locally convex subcomplex, then we can talk about $Y$ intersecting/osculating/inter-osculating with a hyperplane of $X$ by applying the above definitions to the inclusion $Y\xhookrightarrow{}X$.
\end{defn}

\begin{figure}[H]
	\centering
	\scalebox{0.8}{
	\begin{tikzpicture}[auto,node distance=2cm,
		thick,every node/.style={circle,draw,font=\small},
		every loop/.style={min distance=2cm},
		hull/.style={draw=none},
		]
		\tikzstyle{label}=[draw=none,font=\huge]
		
		\begin{scope}
			\draw (4,2)--(0,2)--(2,4)--(4,2);
			\node[fill] (y) at (4,2){};
			\node[label] (y) at (4.1,1.4) {$y$};
			\node[red,fill] (hyp) at (3,3){};
			\node[label,red] (H') at (3.5,3.5) {$H'$};
			\node[hull] (Y1) at (4.5,3){};
			\node[label] (Y) at (0,5) {$Y$};
			
			\node[label] (phi) at (5.5,3.5) {$\phi$};
		\end{scope}
		
		\begin{scope}[shift={(7,0)}]
			\draw[fill=gray!20] (2,4)--(2,6)--(8,6)--(8,0)--(4,0)--(4,2)--(6,2)--(6,4)
			--(4,4)--(4,2)--(2,4);
			\draw (4,2)--(0,2)--(2,4)--(4,2);
			\draw (4,2)--(4,4)--(2,6)--(2,4);
			\draw (4,4)--(6,4)--(8,6)--(2,6);
			\draw (6,4)--(6,2)--(8,0)--(8,6);
			\draw (6,2)--(4,2)--(4,0)--(8,0);
			
			\draw[red,ultra thick] (3,3)--(3,5)--(7,5)--(7,1)--(4,1);
			\node[label,red] (H) at (5,5.5) {$H$};
			\draw[ultra thick] (4,0)--(4,2);
			\node[label] (e) at (3.7,1) {$e$};
			\node[fill] (x) at (4,2) {};
			\node[label] (X) at (1,5) {$X$};
			\node[hull] (X1) at (-.5,3){};
		\end{scope}
		
		\draw[draw=black,fill=blue,-triangle 90, ultra thick] (Y1) -- (X1);
		
	\end{tikzpicture}
}
\caption{$H=\phi[H']$, so $H$ and $Y$ intersect. $H$ and $Y$ also osculate at $(y;e)$.}
\end{figure}
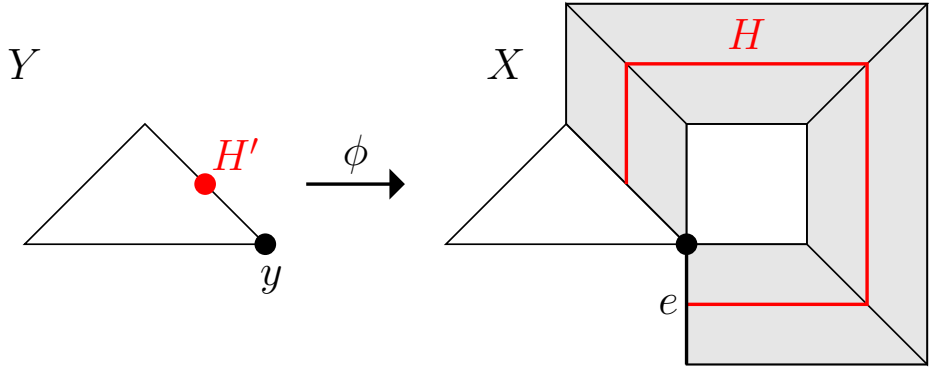

\begin{defn}(Hyperplane carriers)\\
	Let $H$ be a hyperplane in a directly special cube complex $X$. The \emph{carrier of $H$}, denoted $N(H)$, is the smallest subcomplex of $X$ containing $H$. If $X$ is not directly special, one can still define the carrier of $H$ as a certain immersion $N(H)\to X$ (see \cite{WiseRiches}), but this is more delicate and will not concern us.
\end{defn}

\begin{remk}
	Let $H$ be a hyperplane in a directly special cube complex $X$. One can readily verify from the fact that $H$ does not self-intersect or self-osculate that the map $H\times[-1,1]\to X$ from Definition \ref{defn:twosided} is an embedding with image $N(H)$. In particular, the inclusion $H\xhookrightarrow{} N(H)$ is a homotopy equivalence. 
\end{remk}
\begin{remk}\label{remk:CarrierIntersectOsculate}
	In a directly special cube complex $X$, a pair of hyperplanes $H_1$ and $H_2$ intersect (resp. osculate) if and only if $H_1$ and $N(H_2)\to X$ intersect (resp. osculate) -- in fact this remains true in arbitrary cube complexes with the more general definition of hyperplane carrier. Also, the carriers $N(H_1)$ and $N(H_2)$ are disjoint if and only if $H_1$ and $H_2$ neither intersect nor osculate.
\end{remk}

\subsection{Elevations}

\begin{defn}(Elevations)\\\label{defn:elevations}
	Let $\phi:Y\to X$ be a local isometry and $\mu:\hat{X}\to X$ a finite cover. We say that a map $\hat{\phi}:\hat{Y}\to\hat{X}$ is an \emph{elevation} of $\phi$ to $\hat{X}$ if there exists a covering $\nu:\hat{Y}\to Y$ fitting into the commutative diagram
	\begin{equation}\label{elevation}
		\begin{tikzcd}[
			ar symbol/.style = {draw=none,"#1" description,sloped},
			isomorphic/.style = {ar symbol={\cong}},
			equals/.style = {ar symbol={=}},
			subset/.style = {ar symbol={\subset}}
			]
			\hat{Y}\ar{d}{\nu}\ar{r}{\hat{\phi}}&\hat{X}\ar{d}{\mu}\\
			Y\ar{r}{\phi}&X,
		\end{tikzcd}
	\end{equation}
	and such that a path in $\hat{Y}$ closes up as a loop if and only if its projections to $Y$ and $\hat{X}$ both close up as loops. The map $\hat{\phi}$ is necessarily a local isometry. Equivalently, $\hat{Y}$ is a component of the pullback of $\phi$ and $\mu$. If (\ref{elevation}) is a diagram of based spaces, then we call it a \emph{based elevation}.
\end{defn}

\begin{remk}
	Based elevations correspond to intersecting subgroups. Indeed, if (\ref{elevation}) is a diagram of based spaces with respect to basepoints $y\in Y$, $\hat{y}\in\hat{Y}$, $x\in X$ and $\hat{x}\in\hat{X}$, then
	$$(\phi\nu)_*\pi_1(\hat{Y},\hat{y})= \phi_*\pi_1(Y,y)\cap\mu_*\pi_1(\hat{X},\hat{x}).$$
	And choosing a different elevation corresponds to conjugating. Indeed if $\hat{\phi}(\hat{y})\neq \hat{x}$ but instead $\hat{\gamma}$ is a path in $\hat{X}$ from $\hat{x}$ to $\hat{\phi}(\hat{y})$, then
	$$(\phi\nu)_*\pi_1(\hat{Y},\hat{y})=\gamma \phi_*\pi_1(Y,y)\gamma^{-1}\cap\mu_*\pi_1(\hat{X},\hat{x}),$$
	where $\gamma=\mu\hat{\gamma}$.
\end{remk}

\subsection{Cubulated groups and convex subgroups}\label{subsec:cubulated}

\begin{defn}(Cubulated and virtually special groups)\\\label{defn:cubulated}
	A \emph{cubulated group} $G\acts X$ is a finitely generated group $G$ together with a geometric action on a CAT(0) cube complex $X$ by cubical automorphisms. 	
	We say that $G\acts X$ is \emph{virtually special} if there is a finite-index subgroup $\hat{G}<G$ acting freely on $X$ with special quotient $X/\hat{G}$.
\end{defn}

\begin{remk}\label{remk:torsionfree}
In this paper we will usually consider a finite non-positively curved cube complex $X$ with fundamental group $G$, in this case $G\acts\widetilde{X}$ is a cubulated group, where $\widetilde{X}$ is the universal cover of $X$ and $G$ acts on $\widetilde{X}$ by deck transformations. Note that $G$ is torsion-free since any torsion element would act on $\widetilde{X}$ with bounded orbits and fix a point by \cite[II.2.7]{BridsonHaefliger99}.
\end{remk}

\begin{defn}(Convex subgroups and elements)\\\label{defn:convexsubgroups}
Let $G\acts X$ be a cubulated group. A subgroup $K<G$ is \emph{convex} if it stabilises a convex subcomplex $Y\subset X$ with finite quotient $Y/K$ (referred to as a convex-cocompact subgroup by some authors). An element $g\in G$ is \emph{convex} if $\langle g\rangle <G$ is convex.
\end{defn}

\begin{remk}\label{remk:convexsubgp}
If $G=\pi_1(X,x)$ is the fundamental group of a finite non-positively curved cube complex $X$, then $K<G$ being convex is equivalent to $K$ being the image of a homomorphism $\phi_*:\pi_1(Y,y)\xhookrightarrow{}G$, defined by a local isometry $\phi:Y\to X$, with $Y$ finite, and (a homotopy class of) a path $\gamma$ in $X$ from $x$ to $\phi(y)$. We recover the first definition with respect to a based universal cover $(\widetilde{X},\widetilde{x})\to (X,x)$ by letting $\widetilde{\gamma}$ be a lift of $\gamma$ from $\widetilde{x}$ to $\widetilde{y}$, and setting $\widetilde{Y}$ to be the based elevation of $\phi:Y\to X$ with respect to basepoints $y,\phi(y)$ and $\widetilde{y}$ (note that $\widetilde{Y}\to\widetilde{X}$ is an embedding because local isometries of CAT(0) cube complexes are embeddings). The other elevations of $Y$ are stabilised by conjugates of $K$.
\end{remk}

One can eliminate the need for the path $\gamma$ in Remark \ref{remk:convexsubgp} by the following lemma.

\begin{lem}\label{lem:convexisbased}
	Let $X$ be a finite non-positively curved cube complex and let $K<G:=\pi_1(X,x)$ be a convex subgroup. Then there is a based local isometry of finite cube complexes $\phi:(Y,y)\to (X,x)$ with $K=\phi_*\pi_1(Y,y)$.
\end{lem}
\begin{proof}
	Let $(\widetilde{X},\widetilde{x})\to(X,x)$ be the universal cover. Then $G$ acts on $\widetilde{X}$ by deck transformations, and there is a convex subcomplex $\widetilde{Y}\subset\widetilde{X}$ stabilised by $K$ with finite quotient $\widetilde{Y}/K$. By replacing $\widetilde{Y}$ with a cubical thickening as in \cite[Section 4]{HaglundWise12}, we can assume that $\widetilde{x}\in\widetilde{Y}$. Then we can put $(Y,y):=(\widetilde{Y},\widetilde{x})/K$, and define $\phi$ as the quotient map to $(X,x)=(\widetilde{X},\widetilde{x})/G$.
\end{proof}

\bigskip
\section{Imitator homomorphisms}\label{sec:hom}

In this section we describe the construction of walker and imitator, and use it to define imitator homomorphisms. 
The construction is essentially equivalent to the canonical completion and retraction of Haglund--Wise \cite{HaglundWise08} -- the imitator homomorphisms we obtain are precisely the homomorphisms induced by the canonical retraction (restricted to a certain component of the canonical completion). But our interpretation does give new insights and facilitates new applications, as we shall see in subsequent sections.

\begin{cons}(Walker and imitator)\\\label{cons:imitator}
	Let $\phi:Y\to X$ be a local isometry of directly special cube complexes. We consider two people wandering around the 1-skeleta of $X$ and $Y$: the \emph{walker} wanders around $X^1$ while the \emph{imitator} wanders around $Y^1$. The walker wanders freely around $X^1$, while the imitator tries to ``copy'' the walker in the following way: if the imitator and walker are at vertices $(y,x)\in Y^0\times X^0$ and the walker traverses the edge $e\in \link(x)$, then the imitator traverses the edge $f\in\link(y)$ with $\phi(f)\parallel e$, if such an edge exists, otherwise they remain at $y$. Note that the edge $f$, if it exists, will be unique because $H(e)$ doesn't self-osculate at $\phi(y)$. If the walker immediately returns along $e$, then the imitator will return along $f$ in the case that $f$ exists, otherwise they remain at $y$ again, so the whole process is reversible.
	
	Iterating the process, if the walker travels along a path $\gamma$, starting at $x$, then the imitator travels along a path that we denote $\delta=\delta(\gamma,y)$, starting at $y$. Since the process is reversible, if the walker immediately backtracks along the path $\gamma$ then the imitator will backtrack along the path $\delta$. See Figure \ref{fig:imitator} for some examples.
\end{cons}

\begin{figure}[H]\label{fig:imitator}
	\centering
	\scalebox{0.8}{
		\begin{tikzpicture}[auto,node distance=2cm,
			thick,every node/.style={circle,draw,font=\small},
			every loop/.style={min distance=2cm},
			hull/.style={draw=none},
			]
			\tikzstyle{label}=[draw=none,font=\huge]
			
			\begin{scope}[shift={(-2,2)}]
				\draw (0,0)--(0,4)--(4,4)--(4,0)--(0,0);
				
				\path (0,4) edge [blue,very thick,postaction={decoration={markings,mark=at position 0.65 with {\arrow[blue,line width=1mm]{triangle 60}}},decorate}] (0,0);
				\path (0,4) edge [blue,very thick,postaction={decoration={markings,mark=at position 0.65 with {\arrow[blue,line width=1mm]{triangle 60}}},decorate}] (4,4);
				\node[label,blue] at (1.5,2) {$\delta(\gamma_2,y)$};
				\node[label,blue] at (1.8,4.8){$\delta(\gamma_1,y)$};
				\node[fill] (y) at (0,4){};
				\node[label] at (-.5,4.5) {$y$};
				\node[fill] (y') at (0,0){};
				\node[label] at (-.5,-.5){$y'$};
				\node[hull] (Y1) at (4.5,2){};
				\node[label] (Y) at (2,7.5) {$Y$};
				
				\node[label] (phi) at (5.5,2.5) {$\phi$};
			\end{scope}
			
			\begin{scope}[shift={(5,0)}]
				\draw[fill=gray!20] (2,8)--(8,8)--(8,0)--(2,0)--(2,2)
				--(6,2)--(6,6)--(2,6)--(2,8);
				\draw (2,8)--(8,8)--(6,6)--(2,6)--(2,8);
				\draw (8,8)--(8,0)--(6,2)--(6,6);
				\draw (8,0)--(2,0)--(2,2)--(6,2);
				\draw (2,6)--(2,2)--(0,4)--(2,6);
				
				\draw[Green,very thick] (2,8)--(8,8)--(8,0);
				\draw[Green,very thick] (2,6)--(2,2)--(0,4)--(2,6);
				\node[fill] at (2,6){};
				\path (0,4) edge [Green,postaction={decoration={markings,mark=at position 0.5 with {\arrow[Green,line width=1mm]{triangle 60}}},decorate}] (2,2);
				\path (8,0) edge [Green,postaction={decoration={markings,mark=at position 0.8 with {\arrow[Green,line width=1mm]{triangle 60}}},decorate}] (8,8);
				
				\node[label] at (1.5,6.5) {$x$};
				\node[label,Green] at (1,2.5) {$\gamma_2$};
				\node[label,Green] at (8.5,5){$\gamma_1$};
				\node[label] (X) at (4,9.5) {$X$};
				\node[hull] (X1) at (-.5,4){};
			\end{scope}
			
			\draw[draw=black,fill=blue,-triangle 90, ultra thick] (Y1) -- (X1);
			
		\end{tikzpicture}
	}
	\caption{Some possible movements of walker and imitator. }
\end{figure}
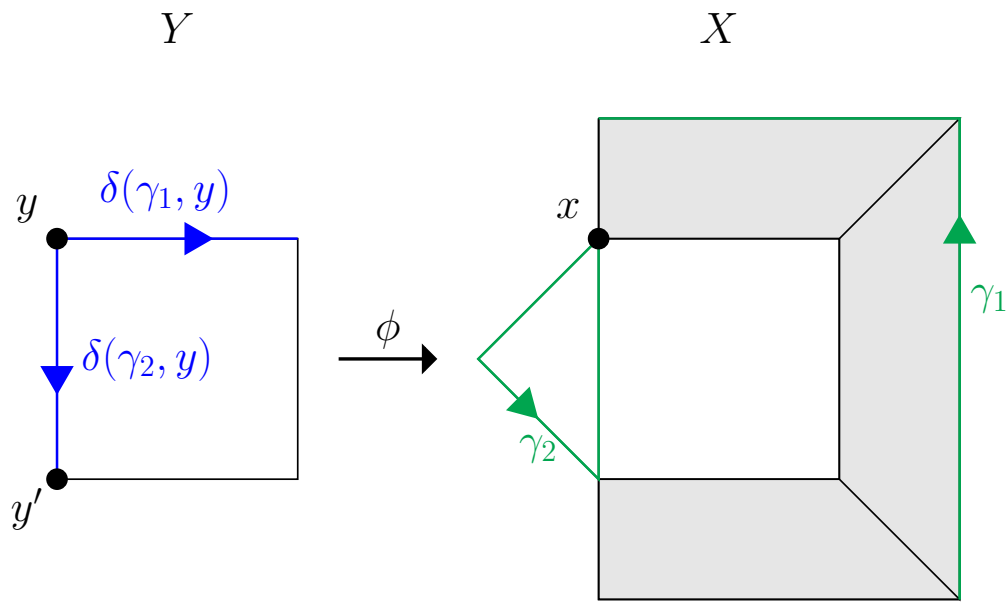

The movements of walker imitator respect homotopies, as we shall see in the following lemma. When we say that two paths are \emph{homotopic} we mean that they have the same endpoints and that they are homotopic relative to these endpoints. We denote the homotopy class of a path $\gamma$ by $[\gamma]$.

\begin{lem}\label{lem:gammasquare}
		If $\gamma$ and $\gamma'$ are homotopic paths in $X$, then for any $y\in Y^0$ the paths $\delta(\gamma,y)$ and $\delta(\gamma',y)$ are homotopic in $Y$.
\end{lem}
\begin{proof}
By definition we are considering the walker starting at a vertex $x$ (the common starting vertex of $\gamma$ and $\gamma'$) and traversing either $\gamma$ or $\gamma'$, and the imitator starting at $y$ and traversing either $\delta(\gamma,y)$ or $\delta(\gamma',y)$.
To prove the lemma it is enough to consider $\gamma$ and $\gamma'$ that differ by a basic homotopy move: so $\gamma$ is obtained from $\gamma'$ either by removing a backtrack along an edge or by pushing a subpath across a square. As discussed in Construction \ref{cons:imitator}, the imitator backtracks whenever the walker backtracks, so the lemma holds if $\gamma$ is obtained from $\gamma'$ by removing a backtrack along an edge. For the rest of the proof it suffices to consider paths $\gamma$ and $\gamma'$ that consist of edges $e_1,e_2$ and $e'_2,e'_1$ respectively, such that $e_1,e_2,e'_1,e'_2$ is the boundary cycle of a square in $X$. Let $x\in X^0$ be incident to $e_1$ and $e'_2$, and let $y\in Y^0$.  Let $H_1:=H(e_1)=H(e'_1)$ and $H_2:=H(e_2)=H(e'_2)$. Note that $H_1$ and $H_2$ intersect at $(x;e_1,e'_2)$, so $H_1\neq H_2$. The proof splits into four cases:
\begin{enumerate}
	\item If no $f\in\link(y)$ has $\phi(f)$ dual to $H_1$ or $H_2$ then $\delta(\gamma,y)=\delta(\gamma',y)$ is the constant path at $y$.
	
	\item\label{item:2} Suppose $f_1\in\link(y)$ has $H(\phi(f_1))=H_1$, but no $f_2\in\link(y)$ has $H(\phi(f_2))=H_2$. Let $y'\in Y^0$ be the vertex at the other end of $f_1$. 
	
	We now argue that no $f_2\in\link(y')$ has $H(\phi(f_2))=H_2$. Indeed if such an $f_2$ did exist then $H_1$ and $H_2$ would intersect at $(\phi(y');\phi(f_1),\phi(f_2))$ (they can't osculate here because they already intersect at $(x;e_1,e'_2)$, and inter-osculating is forbidden), and so $\phi(f_1), \phi(f_2)$ would form the corner of a square at $\phi(y')$; but then local injectivity of $\phi$ implies $f_1,f_2$ would form the corner of a square $S$ at $y'$, and the edge $f'_2\in\link(y)$ opposite $f_2$ in $S$ would have $H(\phi(f'_2))=H_2$, a contradiction.
	
	Now consider the walker traversing $\gamma$: we see that the imitator traverses $f_1$ as the walker traverses $e_1$ and stays at $y'$ as the walker traverses $e_2$. So $\delta(\gamma,y)$ just consists of the edge $f_1$. And if the walker traverses $\gamma'$: the imitator stays at $y$ as the walker traverses $e'_2$ and traverses $f_1$ as the walker traverses $e'_1$. So $\delta(\gamma',y)$ just consists of the edge $f_1$ as well.
	
	\item If no $f_1\in\link(y)$ has $H(\phi(f_1))=H_1$ but some $f_2\in\link(y)$ has $H(\phi(f_2))=H_2$, then we can deduce in a similar manner to the previous case that $\delta(\gamma,y)=\delta(\gamma',y)$ just consists of the edge $f_2$.
	\item Finally, suppose $f_1\in\link(y)$ has $H(\phi(f_1))=H_1$ and $f_2\in\link(y)$ has $H(\phi(f_2))=H_2$. Then $H_1$ and $H_2$ intersect at $(\phi(y);\phi(f_1),\phi(f_2))$, and local injectivity of $\phi$ implies that $f_1,f_2$ form the corner of a square $S$ at $y$. It follows that $\delta(\gamma,y)$ consists of two sides of $S$, starting with $f_1$, whilst $\delta(\gamma',y)$ also consists of two sides of $S$, but starting with $f_2$, and both paths end at the vertex of $S$ opposite $y$. In particular, $\delta(\gamma',y)$ is obtained from $\delta(\gamma,y)$ by homotoping across $S$.\qedhere
\end{enumerate}
\end{proof}

\begin{figure}[H]
	\centering
	\scalebox{0.8}{
\begin{tikzpicture}[auto,node distance=2cm,
	thick,every node/.style={circle,draw,font=\small},
	every loop/.style={min distance=2cm},
	hull/.style={draw=none},
	]
	\tikzstyle{label}=[draw=none,font=\huge]
	
	\begin{scope}[shift={(2,0)}]
		\path (0,0) edge [blue,very thick,postaction={decoration={markings,mark=at position 0.65 with {\arrow[blue,line width=1mm]{triangle 60}}},decorate}] (0,3);
		\node[fill] (y) at (0,0) {};
		\node[fill] (y') at (0,3) {};
		\node[label] (y1) at (0,-0.7) {$y$};
		\node[label] (y'1) at (0,3.7) {$y'$};
		\node[label] (f1) at (-.6,1.5) {$f_1$};
		\node[label,blue,text width=3cm] at (2,1.5) {$\delta(\gamma,y)=\delta(\gamma',y)$};
	\end{scope}
	
	\begin{scope}[shift={(9,0)}]
		\draw[fill=gray!20] (0,0)--(0,3)--(3,3)--(3,0)--(0,0);
		\draw[Green,very thick] (0,0)--(0,3)--(3,3)--(3,0)--(0,0);
		\path (0,3) edge [Green,postaction={decoration={markings,mark=at position 0.25 with {\arrow[Green,line width=1mm]{triangle 60}}},decorate}] (3,3);
		\path (3,0) edge [Green,postaction={decoration={markings,mark=at position 0.25 with {\arrow[Green,line width=1mm]{triangle 60}}},decorate}] (3,3);
		\node[fill] (x) at (0,0) {};
		\node[fill] (end) at (3,3){};
		\node[label,Green] (ga) at (1.3,2.5) {$\gamma$};
		\node[label,Green] (ga') at (2.5,1.3) {$\gamma'$};
		\node[label] (e1) at (-.5,1.5) {$e_1$};
		\node[label] (e2) at (1.5,3.5) {$e_2$};
		\node[label] (e'1) at (3.5,1.5) {$e'_1$};
		\node[label] (e'2) at (1.5,-.5) {$e'_2$};
	\end{scope}
	
\end{tikzpicture}
}
\caption{Case \ref{item:2} of the proof of Lemma \ref{lem:gammasquare}.}
\end{figure}
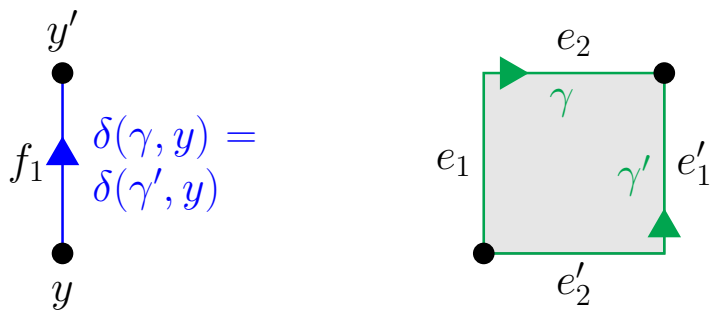

\begin{cons}(Imitator homomorphisms)\\\label{cons:hom}
	Let $\phi:Y\to X$ be a local isometry of directly special cube complexes and pick a basepoint $x\in X^0$. We have a right-action of $G:=\pi_1(X,x)$ on $Y^0$ defined as follows: for $y\in Y^0$ and $[\gamma]\in G$, let $y\cdot[\gamma]$ be the endpoint of the path $\delta(\gamma,y)$ -- this only depends on the homotopy class of $\gamma$ by Lemma \ref{lem:gammasquare}. Let $G_y<G$ be the stabiliser of $y\in Y^0$, which has finite index in $G$ if $Y$ is finite. We can then define a homomorphism $\rho_{x,y}:G_y\to\pi_1(Y,y)$ by $\rho_{x,y}([\gamma])=[\delta(\gamma,y)]$ -- again this is well-defined by Lemma \ref{lem:gammasquare}.
	
	Throughout this paper we will usually work with a based local isometry $\phi:(Y,y)\to(X,x)$, and consider $\pi_1(Y,y)$ as a subgroup of $G=\pi_1(X,x)$. We will call $G_\phi:=G_y<G$ the \emph{imitator subgroup of $\phi$}, and call the corresponding cover of $X$ the \emph{imitator cover of $\phi$}. We will refer to the homomorphism  $\rho_\phi:=\rho_{x,y}:G_\phi\to\pi_1(Y,y)<G$ as the \emph{imitator homomorphism of $\phi$}. For $[\gamma]\in\pi_1(Y,y)$ the $\phi$-images of the edges of $\gamma$ are equal to, hence parallel to, the edges of $\phi\gamma$, so it follows that $\delta(\phi\gamma,y)=\gamma$ (i.e. as the walker travels along $\phi\gamma$, the imitator shadows the walker by travelling along $\gamma$). Consequently, $\rho_\phi$ is a group retraction onto the subgroup $\pi_1(Y,y)<G_y$ (i.e. $\rho_\phi$ restricts to the identity map on $\pi_1(Y,y)$).
	
	An example of a based local isometry $\phi:(Y,y)\to(X,x)$ is in Figure \ref{fig:imitator}; note that $\delta(\gamma_2,y')=\delta(\gamma_2,y)^{-1}$, so $[\gamma_2]\notin G_\phi$ but $[\gamma_2]^2\in G_\phi$ -- in fact $[\gamma_2]^2\in\ker\rho_\phi$.
\end{cons}

\begin{remk}
	In Section \ref{sec:imitatorcovers} we give a description of the vertices and edges of the imitator cover directly in terms of the movements of walker and imitator. This viewpoint establishes the equivalence of our construction with the canonical completion and retraction of Haglund--Wise -- in particular it shows that the imitator homomorphism is induced by a cellular retraction of cube complexes.
\end{remk}

\bigskip
\section{First applications}\label{sec:applysingle}

In this section we apply the construction of imitator homomorphisms and imitator covers to prove various propositions and lemmas, some new and others already known, culminating with Proposition \ref{prop:triplecoset} about separability of triple cosets.

\subsection{Subgroup separability}

\begin{defn}(Separability)\\\label{defn:separability}
	A subgroup $K<G$ is \emph{separable} if it is an intersection of finite-index subgroups in $G$. $G$ is \emph{residually finite} if the trivial subgroup is separable in $G$. The collection of all cosets of finite-index subgroups of $G$ is a basis for the \emph{profinite topology} on $G$. A subset $A\subset G$ is \emph{separable} if it is closed in the profinite topology (note that these two notions of separability are equivalent for subgroups). 
\end{defn}

In \cite{HaglundWise08} it was shown that fundamental groups of finite special cube complexes are virtually subgroups of right angled Coxeter groups, hence they are $\Z$-linear and residually finite. In fact residual finiteness can be deduced rather more directly using imitator subgroups.

\begin{prop}\label{prop:residuallyfinite}
	Fundamental groups of finite virtually special cube complexes are residually finite.
\end{prop}
\begin{proof}
	Residual finiteness is a commensurability invariant, so it suffices to work with a finite directly special cube complex $X$. Let $x\in X^0$ be a basepoint, and let $1\neq[\gamma]\in G=\pi_1(X,x)$ be an essential loop based at $x$. 	
	Let $(\widetilde{X},\widetilde{x})\to (X,x)$ be the based universal cover of $X$ and let $\widetilde{\gamma}$ be the based lift of $\gamma$. Let $Y_{\widetilde{\gamma}}\subset\widetilde{X}$ be the smallest convex subcomplex of $\widetilde{X}$ containing $\widetilde{\gamma}$. Note that  $Y_{\widetilde{\gamma}}$ is obtained by intersecting all halfspaces in $\widetilde{X}$ containing $\widetilde{\gamma}$, and that a hyperplane of $\widetilde{X}$ intersects $Y_{\widetilde{\gamma}}$ if and only if it intersects $\widetilde{\gamma}$ -- there are only finitely many such hyperplanes, so $Y_{\widetilde{\gamma}}$ is finite. The restriction of the covering to $Y_{\widetilde{\gamma}}$ is a local isometry $\phi:(Y_{\widetilde{\gamma}},\widetilde{x})\to (X,x)$, so we may apply the walker-imitator construction. Since $\gamma$ lifts to $\widetilde{\gamma}$, we have $\delta(\gamma,\widetilde{x})=\widetilde{\gamma}$. But $\widetilde{\gamma}$ has distinct endpoints, so $[\gamma]$ is not contained in the finite-index subgroup $G_\phi$, as required.
\end{proof}

We can deduce subgroup separability using imitator homomorphisms. The analogous result of Haglund--Wise is \cite[Corollary 7.9]{HaglundWise08}.

\begin{prop}\label{prop:subgroupseparability}
	Let $X$ be a finite virtually special cube complex. Then any convex subgroup of $\pi_1X$ is separable.
\end{prop}
\begin{proof}
	By Lemma \ref{lem:convexisbased} we may consider a convex subgroup defined by a based local isometry $\phi:(Y,y)\to (X,x)$. Take a finite directly special cover $(\hat{X},\hat{x})\to(X,x)$ and a based elevation $\hat{\phi}:(\hat{Y},\hat{y})\to(\hat{X},\hat{x})$ of $\phi$. $\pi_1(\hat{Y},\hat{y})$ is a finite union of cosets of $\pi_1(Y,y)$ (viewing both as subgroups of $\pi_1(X,x)$), and separability is equivalent to being closed in the profinite topology (Definition \ref{defn:separability}), so it suffices to prove that $\pi_1(\hat{Y},\hat{y})$ is separable.
	Put $\hat{G}:=\pi_1(\hat{X},\hat{x})$ and let $\rho_{\hat{\phi}}:\hat{G}_{\hat{y}}\to\pi_1(\hat{Y},\hat{y})$ be the imitator homomorphism of $\hat{\phi}$, which is a retraction by Construction \ref{cons:hom}. These groups are residually finite by Proposition \ref{prop:residuallyfinite}, so it follows from Lemma \ref{lem:separablesubgroup} below that $\pi_1(\hat{Y},\hat{y})$ is separable in $\hat{G}_{\hat{y}}$, and hence also in $\pi_1(X,x)$ (finite-index subgroups are always closed in the profinite topology).
\end{proof}

The following lemma is proved for instance in \cite[Lemma 3.9]{HsuWise99}, our proof is slightly different and leads to generalisations with more subgroups (Lemmas \ref{lem:separabledcoset} and \ref{lem:separabletcoset}).

\begin{lem}\label{lem:separablesubgroup}
	Let $\rho:G\twoheadrightarrow K< G$ be a retraction of a residually finite group $G$. Then $K$ is separable in $G$.
\end{lem}
\begin{proof}
	Let $g\in G-K$. The goal is to find a finite quotient of $G$ that separates $g$ from $K$. We know that $\rho(g)\in K$, so $g\neq\rho(g)$, and by residual finiteness of $G$ there exists a finite quotient $q:G\to\bar{G}$ with
	\begin{equation}\label{qequation1}
		q(g)\neq q\rho(g).
	\end{equation} 	
	Now consider another finite quotient of $G$, defined by
	\begin{align*}
		t:G&\to\bar{G}\times\bar{G}\\
		h&\mapsto(q(h),q\rho(h)).
	\end{align*}
	We claim that $t(g)\notin t(K)$, so suppose for contradiction that $t(g)=t(k)$ with $k\in K$. Then
	\begin{equation}\label{texpand1}
		(q(g),q\rho(g))=(q(k),q\rho(k)).
	\end{equation}
	But $\rho(k)=k$, so (\ref{texpand1}) implies that $q(g)=q\rho(g)$, contradicting (\ref{qequation1}).
\end{proof}
\bigskip

\subsection{Entrapment}

The following will be a key lemma used throughout the paper.

\begin{lem}(Subcomplex Entrapment)\\\label{lem:subcomplextrap}
	Let $\phi:Y\to X$ be a local isometry of directly special cube complexes, and let $Z\subset X$ be a locally convex subcomplex that does not inter-osculate with hyperplanes of $X$. If we consider an imitator in $Y$ and a walker in $X$ starting at positions  $(y,x)\in\phi^{-1}(Z)\times Z$, then the imitator stays inside $\phi^{-1}(Z)$ as long as the walker stays inside $Z$. In particular, if $Y$ is a locally convex subcomplex of $X$, then the imitator stays inside $Y\cap Z$.	
\end{lem}
\begin{proof}
	It suffices to consider the walker traversing a single edge $e\in\link(x)$ inside $Z$. If the imitator stays at $y$ then we are done, otherwise they traverse an edge $f$ with $\phi(f)\parallel e$. 
	Suppose for contradiction that $f$ is not in $\phi^{-1}(Z)$. Then $H(\phi(f))=H(e)$ osculates with $Z$ at $(\phi(y);\phi(f))$. But $H(e)$ and $Z$ also intersect because $e$ is in $Z$, hence they inter-osculate, contrary to the assumption of the lemma.
\end{proof}
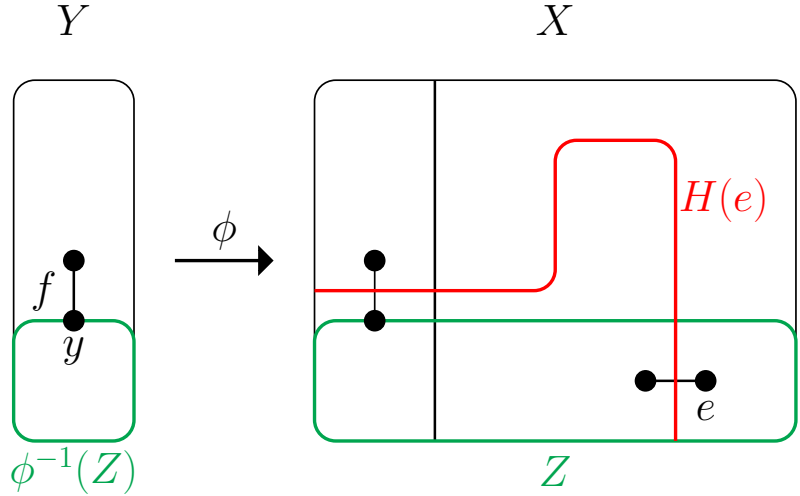
\begin{figure}[H]
	\centering
	\scalebox{0.8}{
\begin{tikzpicture}[auto,node distance=2cm,
	thick,every node/.style={circle,draw,font=\small},
	every loop/.style={min distance=2cm},
	hull/.style={draw=none},
	]
	\tikzstyle{label}=[draw=none,font=\huge]
	
	\begin{scope}
		\draw[rounded corners=10pt] (0,0) rectangle (2,6);
		\draw[rounded corners=10pt,Green,ultra thick] (0,0) rectangle (2,2);
		\draw[very thick] (1,2)--(1,3);
		\node[fill] at (1,2){};
		\node[fill] at (1,3){};
		\node[label] at (.5,2.5) {$f$};
		\node[label,Green] at (1,-.5) {$\phi^{-1}(Z)$};
		\node[hull] (Y1) at (2.5,3){};
		\node[label] (Y) at (1,7) {$Y$};
		\node[label] at (1,1.5) {$y$};
		
		\node[label] (phi) at (3.5,3.5) {$\phi$};
	\end{scope}
	
	\begin{scope}[shift={(5,0)}]
		\draw[rounded corners=10pt] (0,0) rectangle (8,6);
		\draw[very thick] (2,0)--(2,6);
		\draw[rounded corners=10pt,Green,ultra thick] (0,0) rectangle (8,2);
		\draw (1,2)--(1,3);
		\node[fill] at (1,2){};
		\node[fill] at (1,3){};
		
		\draw[very thick] (5.5,1)--(6.5,1);
		\node[fill] at (5.5,1){};
		\node[fill] at (6.5,1){};
		\node[label] at (6.5,.5) {$e$};
		
		\draw[rounded corners=10pt,red,ultra thick] (0,2.5)--(4,2.5)--(4,5)
		--(6,5)--(6,0);
		\node[label,red] at (6.8,4) {$H(e)$};
		
		\node[label] (X) at (4,7) {$X$};
		\node[hull] (X1) at (-.5,3){};
		\node[label,Green] at (4,-.5) {$Z$};
	\end{scope}
	
	\draw[draw=black,fill=blue,-triangle 90, ultra thick] (Y1) -- (X1);
	
\end{tikzpicture}
}
\caption{Proof of Lemma \ref{lem:subcomplextrap}.}
\end{figure}

The following special case of Subcomplex Entrapment allows us to trap imitators inside hyperplane carriers.

\begin{lem}(Hyperplane Entrapment)\\\label{lem:hyptrap}
Let $\phi:Y\to X$ be a local isometry of directly special cube complexes. Let $H$ be a hyperplane in $Y$, and suppose that the imitator and walker start at positions $(y,x)\in N(H)\times N(\phi[H])$. If the walker stays inside $N(\phi[H])$ then the imitator stays inside $N(H)$. 
\end{lem}
\begin{proof}
	We apply Subcomplex Entrapment with $Z=N(\phi[H])$. This ensures that the imitator stays inside $\phi^{-1}(N(\phi[H]))$ as long as the walker stays inside $N(\phi[H])$. But we know that
	$$\phi^{-1}(N(\phi[H]))=\bigcup_{\phi[H']=\phi[H]}N(H'),$$
	and this is a disjoint union because any vertex in an intersection of hyperplane carriers would map to a point of self-intersection or self-osculation of $\phi[H]$. Hence the imitator remains in $N(H)$.
\end{proof}

As an application of Hyperplane Entrapment we have the following proposition about imitator covers. 

\begin{prop}\label{prop:completionhyps}
	Let $\phi:(Y,y)\to (X,x)$ be a local isometry of directly special cube complexes and let $(\dot{X},\dot{x})\to(X,x)$ be the associated imitator cover. We know from Construction \ref{cons:hom} that $\pi_1(Y,y)<G_\phi=\pi_1(\dot{X},\dot{x})$, so there is a lift $j:(Y,y)\to(\dot{X},\dot{x})$ of $\phi$. Then we have the following:
	\begin{enumerate}
		\item\label{item:embedding} $j:(Y,y)\to(\dot{X},\dot{x})$ is an embedding.
		\item\label{item:nointosc} No hyperplane in $\dot{X}$ inter-osculates with $Y$ (with respect to $j$).
	\end{enumerate}
\end{prop}
\begin{proof}$ $\par
	\begin{enumerate}
		\item As $j$ is a local isometry, it's enough to show injectivity of $j$ on $Y^0$.
		Suppose $y_1,y_2\in Y^0$ with $j(y_1)=j(y_2)$.
		Let $\gamma_1,\gamma_2$ be paths in $Y$ from $y$ to $y_1,y_2$ respectively. We know that $j\gamma_1 * (j\gamma_2)^{-1}$ is a loop in $(\dot{X},\dot{x})$, so $\phi\gamma_1 * (\phi\gamma_2)^{-1}$ is a loop in $(X,x)$ whose homotopy class lies in $G_\phi$.
		This means that the imitator traverses a loop in $Y$ based at $y$ as the walker traverses $\phi\gamma_1 * (\phi\gamma_2)^{-1}$. But $\delta(\phi\gamma_1,y)=\gamma_1$ and $\delta(\phi\gamma_2,y)=\gamma_2$ (see Construction \ref{cons:hom}), so the imitator must in fact traverse the loop $\gamma_1 * \gamma_2^{-1}$ -- in particular $\gamma_1,\gamma_2$ have a common endpoint $y_1=y_2$.
		\item Let $H$ be a hyperplane in $Y$, and let $y_1\in N(H)$ be a vertex. Suppose for contradiction that $\dot{H}:=j[H]$ osculates with $Y$ at some pair $(y_2; \dot{e}_2)\in Y^0\times \dot{X}^1$, so $\dot{e}_2$ is incident to $j(y_2)$. Suppose $\dot{e}_2$ projects to $e_2\in X^1$. Let $\dot{\gamma}$ be a path in $N(\dot{H})$ from $j(y_1)$ to $j(y_2)$. Then $\dot{\gamma}$ projects to a path $\gamma$ in $N(\phi[H])\subset X$ from $x_1:=\phi(y_1)$ to $x_2:=\phi(y_2)$.
		Let $\gamma_1,\gamma_2$ be paths in $Y$ from $y$ to $y_1,y_2$ respectively. The loop $j\gamma_1 * \dot{\gamma} * (j\gamma_2)^{-1}$ in $\dot{X}$ projects to the loop $\phi\gamma_1 * \gamma * (\phi\gamma_2)^{-1}$ in $X$, so the latter must have homotopy class in the imitator subgroup $G_\phi$. So the imitator traverses a loop in $Y$ based at $y$ as the walker traverses $\phi\gamma_1 * \gamma * (\phi\gamma_2)^{-1}$. The imitator must shadow the walker for the $\phi\gamma_1$ and $(\phi\gamma_2)^{-1}$ parts of the walkers journey, so the imitator's route takes the form $\gamma_1 * \delta * \gamma_2^{-1}$, where $\delta$ is a path from $y_1$ to $y_2$. As $\gamma$ is inside $N(\phi[H])$, Hyperplane Entrapment implies that $\delta$ is inside $N(H)$, so $y_2\in N(H)$, and there is an edge $f\in\link(y_2)$ dual to $H$.
		Since $\dot{H}$ osculates with $Y$ at $(y_2;\dot{e}_2)$, we know that $j(f)\neq\dot{e}_2$, so projecting to $X$ we get distinct edges $\phi(f),e_2\in\link(x_2)$ that are both dual to $\phi[H]$. Therefore $\phi[H]$ either self-intersects or self-osculates at $x$, a contradiction.\qedhere		
	\end{enumerate}
\end{proof}

We get the following corollary. This will be used throughout the paper to get our hands on subcomplexes that do not inter-osculate with hyperplanes, which then allows us to use Subcomplex Entrapment.

\begin{cor}\label{cor:embedPi}
	Let $Y_1,...,Y_n\to X$ be local isometries of finite virtually special cube complexes. Then there is a finite directly special regular cover $\hat{X}\to X$ such that all elevations of the $Y_i$ to $\hat{X}$ are embedded and do not inter-osculate with hyperplanes of $\hat{X}$.
\end{cor}
\begin{proof}
	Firstly, we may assume that $X$ is already directly special, as otherwise we can pass to a finite directly special cover of $X$ and replace $Y_1,...,Y_n$ with all their elevations to this cover. Next, apply Proposition \ref{prop:completionhyps} to each $Y_i$ in turn to obtain finite covers $X_i\to X$, each containing an elevation of $Y_i$ which is embedded and does not inter-osculate with hyperplanes. The property of an elevation being embedded and not inter-osculating with hyperplanes passes to further covers. The desired cover $\hat{X}\to X$ is any finite regular cover of $X$ that factors through all the covers $X_i\to X$. The regularity of the cover means that one elevation of $Y_i$ will have the desired property if and only if all elevations of $Y_i$ do, so we are done.
\end{proof}

The following proposition is a simple example of the power of Subcomplex Entrapment. This proposition will be generalised to more than two subcomplexes in Section \ref{sec:hierarchical}.

\begin{prop}\label{prop:twoconintersection}
	Let $Y_1$ and $Y_2$ be locally convex subcomplexes of a finite directly special cube complex $X$, and suppose that $Y_2$ does not inter-osculate with any hyperplanes of $X$.
	Suppose $x\in Y_1\cap Y_2$, and let $(\dot{X},\dot{x})\to(X,x)$ be the imitator cover of the inclusion $\phi:(Y_1,x)\xhookrightarrow{} (X,x)$. Then the based elevations of $Y_1$ and $Y_2$ to $(\dot{X},\dot{x})$ have connected intersection.
\end{prop}
\begin{proof}
	Denote the based elevations of $Y_1$ and $Y_2$ by $\dot{Y}_1$ and $\dot{Y}_2$.
	Our task is to show that the only component of $\dot{Y}_1\cap\dot{Y}_2$ is the one containing $\dot{x}$.
	Suppose $\dot{x}'\in\dot{Y}_1\cap\dot{Y}_2$ and suppose $\dot{\gamma}_1,\dot{\gamma}_2$ are paths from $\dot{x}$ to $\dot{x}'$ in $\dot{Y}_1,\dot{Y}_2$ respectively. Say $\dot{\gamma}_1,\dot{\gamma}_2$ descend to paths $\gamma_1,\gamma_2$ in $X$, so these will lie in $Y_1,Y_2$ respectively. Then $[\gamma_1 * \gamma_2^{-1}]$ lies in the imitator subgroup of $\phi$, so the imitator traverses a loop in $Y_1$ based at $x$ as the walker traverses $\gamma_1 * \gamma_2^{-1}$. The imitator must shadow the walker for the $\gamma_1$ part of the walker's journey, so the imitator traverses a loop of the form $\gamma_1 * \gamma'_2$. As no hyperplane inter-osculates with $Y_2$, we may apply Subcomplex Entrapment (Lemma \ref{lem:subcomplextrap}) to deduce that $\gamma'_2$ lies in $Y_1\cap Y_2$. But $\gamma_1 * \gamma'_2$ is a loop that lies entirely in $Y_1$, so $\delta(\gamma_1 * \gamma'_2,x)=\gamma_1 * \gamma'_2$ and $[\gamma_1 * \gamma'_2]$ lies in the imitator subgroup of $\phi$. Hence $\gamma_1 * \gamma'_2$ lifts to a loop $\dot{\gamma}_1 * \dot{\gamma}'_2$ in $\dot{Y}_1$, with $\dot{\gamma}'_2$ a path in $\dot{Y}_1\cap\dot{Y}_2$ from $\dot{x}'$ to $\dot{x}$.
\end{proof}

\subsection{Double coset separability}

Having proven that convex subgroups are separable in Proposition \ref{prop:subgroupseparability}, we now turn our attention to double cosets of convex subgroups. This Proposition was proven by Oreg\'on-Reyes in \cite[Theorem A.1]{Oregonreyes20}; the proof below is similar but recast into the language of imitator homomorphisms.

\begin{prop}\label{prop:doublecoset}
	Let $X$ be a finite virtually special cube complex, and let $K_1,K_2<G:=\pi_1X$ be convex subgroups. Then for any $g\in G$ the double coset $K_1 gK_2$ is separable in $G$.
\end{prop}
\begin{proof}
	 As separability is equivalent to being closed in the profinite topology, left and right translates of separable subsets are still separable. So it is enough to prove that $g^{-1}K_1 gK_2$ is separable. But conjugates of convex subgroups are still convex, so we may assume from the start that $g=1$.
		
	 Choose a basepoint $x\in X$ and write $G=\pi_1(X,x)$. By Lemma \ref{lem:convexisbased} there are based local isometries of finite cube complexes $(Y_i,y_i)\to (X,x)$ inducing the subgroups $K_i$. By Corollary \ref{cor:embedPi}, there is a finite directly special cover $(\hat{X},\hat{x})\to(X,x)$ such that the based elevations $(\hat{Y}_i,\hat{y}_i)\xhookrightarrow{}(\hat{X},\hat{x})$ of $(Y_i,y_i)\to (X,x)$ are embedded and do not inter-osculate with hyperplanes of $\hat{X}$. We will consider the $\hat{Y}_i$ as subcomplexes of $\hat{X}$ with $\hat{y}_i=\hat{x}$. If $\hat{G}<G$ is the subgroup corresponding to $(\hat{X},\hat{x})\to(X,x)$ then $\hat{K}_i:=K_i\cap\hat{G}$ is the subgroup corresponding to $\hat{Y}_i\subset\hat{X}$. And the double coset $K_1 K_2$ is a finite union of translates of double cosets $g_1\hat{K}_1 \hat{K}_2 g_2$ for $g_i\in K_i$, so it is enough to prove that $\hat{K}_1\hat{K}_2$ is separable in $\hat{G}$ (a finite union of closed sets is closed).
		
	 We now consider the imitator homomorphism $\rho_\phi:\hat{G}_{\phi}\to\hat{K}_1<\hat{G}_{\phi}$ of the inclusion $\phi:(\hat{Y}_1,\hat{x})\xhookrightarrow{}(\hat{X},\hat{x})$. Let $\dot{K}_2:=\hat{K}_2\cap\hat{G}_{\phi}$. 
	 If the walker traverses a loop $\gamma$ in $\hat{Y}_2$ based at $\hat{x}$, then Subcomplex Entrapment implies that the imitator stays in $\hat{Y}_1\cap\hat{Y}_2$. It follows that $\rho_\phi(\dot{K}_2)<\dot{K}_2$. We can then apply Lemma \ref{lem:separabledcoset} below to deduce that $\hat{K}_1\dot{K}_2$ is separable in $\hat{G}_{\phi}$. Since $\hat{G}_{\phi}$ has finite index in $\hat{G}$, we obtain separability of $\hat{K}_1\hat{K}_2$ in $\hat{G}$, as required.
\end{proof}

\begin{lem}\label{lem:separabledcoset}
	Let $\rho:G\twoheadrightarrow K_1< G$ be a retraction of a group $G$, and let $K_2<G$ be a separable subgroup with $\rho(K_2)< K_2$. Then the double coset $K_1 K_2$ is separable in $G$.
\end{lem}
\begin{proof}
	Let $g\in G-K_1 K_2$. The goal is to find a finite quotient of $G$ that separates $g$ from $K_1 K_2$. We know that $\rho(g)\in K_1$, so $g\notin\rho(g)K_2$, and by separability of $K_2$ there exists a finite quotient $q:G\to\bar{G}$ with
	\begin{equation}\label{qequation}
q(g)\notin (q\rho(g))q(K_2).
	\end{equation} 	
	Now consider another finite quotient of $G$, defined by
	\begin{align*}
		t:G&\to\bar{G}\times\bar{G}\\
		h&\mapsto(q(h),q\rho(h)).
	\end{align*}
We claim that $t(g)\notin t(K_1 K_2)$, so suppose for contradiction that $t(g)=t(k_1k_2)$ with $k_i\in K_i$. Then
\begin{equation}\label{texpand}
(q(g),q\rho(g))=(q(k_1k_2),q\rho(k_1k_2)).
\end{equation}
But $\rho(k_1)=k_1$ and $\rho(k_2)\in K_2$, so $q(k_1k_2)\in (q\rho(k_1k_2))q(K_2)$, contradicting (\ref{qequation}) and (\ref{texpand}).
\end{proof}

The following proposition is a good example of how double coset separability can be used.

\begin{prop}
	For $i=1,2$, suppose we have a commutative diagram of local isometries of virtually special cube complexes
	\begin{equation}
		\begin{tikzcd}[
			ar symbol/.style = {draw=none,"#1" description,sloped},
			isomorphic/.style = {ar symbol={\cong}},
			equals/.style = {ar symbol={=}},
			subset/.style = {ar symbol={\subset}}
			]
			(\widetilde{Y}_i,\widetilde{y}_i)\ar{d}\ar[hook,r]&(\widetilde{X},\widetilde{x}_i)\ar{d}{\mu_i}\\
			(Y_i,y_i)\ar{r}&(X,x_i),
		\end{tikzcd}
	\end{equation}
	where the vertical maps are universal covering maps, $\mu_1=\mu_2$, and $Y_i$ and $X$ are finite. Let $\gamma$ be a path in $X$ from $x_1$ to $x_2$, and suppose that it lifts to a path $\widetilde{\gamma}$ in $\widetilde{X}$ from $\widetilde{x}_1$ to $\widetilde{x}_2$. If $\widetilde{Y}_1$ and $\widetilde{Y}_2$ are disjoint in $\widetilde{X}$, then there exists a finite cover $\phi:\hat{X}\to X$ and based elevations
	\begin{equation}
		\begin{tikzcd}[
			ar symbol/.style = {draw=none,"#1" description,sloped},
			isomorphic/.style = {ar symbol={\cong}},
			equals/.style = {ar symbol={=}},
			subset/.style = {ar symbol={\subset}}
			]
			(\hat{Y}_i,\hat{y}_i)\ar{d}\ar[hook,r]&(\hat{X},\hat{x}_i)\ar{d}\\
			(Y_i,y_i)\ar{r}&(X,x_i),
		\end{tikzcd}
	\end{equation}
	such that $\gamma$ lifts to a path $\hat{\gamma}$ from $\hat{x}_1$ to $\hat{x}_2$ and so that $\hat{Y}_1$ and $\hat{Y}_2$ are disjoint in $\hat{X}$.
\end{prop}
\begin{proof}
	Let $K_i<G:=\pi_1 X$ be subgroups corresponding to $Y_i\to X$ such that $K_i$ stabilises $\widetilde{Y}_i$. Suppose the $K_i$-translates of the $R$-ball centred at $\widetilde{x}_i$ cover $\widetilde{Y}_i$. We will define the finite cover $\hat{X}\to X$ by picking a finite-index normal subgroup $\hat{G}\triangleleft G$ and taking the quotient $(\hat{X},\hat{x}_i):=(\widetilde{X},\widetilde{x}_i)/\hat{G}$. 
	There are finite covers of $X$ where the elevations of $Y_i$ embed, so the based elevations $\hat{Y}_i$ will embed in $\hat{X}$ if $\hat{X}\to X$ factors through these covers (i.e. if $\hat{G}$ is small enough). Our task is to pick $\hat{G}$ such that $\hat{Y}_1$ and $\hat{Y}_2$ are disjoint in $\hat{X}$, or equivalently such that no $g\in \hat{G}$ satisfies
	\begin{equation}\label{badg}
		\widetilde{Y}_1\cap g\widetilde{Y}_2\neq\emptyset.
	\end{equation}
	If $g\in G$ satisfies (\ref{badg}) then there exists $k_2\in K_2$ with $d(\widetilde{Y}_1,gk_2\widetilde{x}_2)\leq R$ and also $k_1\in K_1$ with $d(\widetilde{x}_1,k_1gk_2\widetilde{x}_2)\leq 2R$. Writing $g':=k_1gk_2$ we know that $g'$ also satisfies (\ref{badg}). However, no element of the double coset $K_1 K_2$ satisfies (\ref{badg}); so $g'\notin K_1 K_2$, and by double coset separability we can choose $\hat{G}$ so that
	\begin{equation}\label{separateg'}
		g'\notin\hat{G}K_1K_2,
	\end{equation} 
	and hence $g\notin\hat{G}$. There are only finitely many $g'\in G$ with $d(\widetilde{x}_1,g'\widetilde{x}_2)\leq 2R$, so we can satisfy (\ref{separateg'}) for all possible $g'$ that arise in this way, thus completing the proof.
\end{proof}

\subsection{Triple coset separability}

We can push the arguments of Proposition \ref{prop:doublecoset} slightly further to also obtain separability of triple cosets. This result is new, although multicoset separability for virtually special hyperbolic groups follows from \cite{Minasyan06}.

\begin{prop}\label{prop:triplecoset}
	Let $X$ be a finite virtually special cube complex, and let $K_1,K_2,K_3<G:=\pi_1X$ be convex subgroups. Then for any $g_1,g_2\in G$ the triple coset $K_1g_1K_2 g_2K_3$ is separable in $G$.
\end{prop}
\begin{proof}
Conjugates of convex subgroups are convex, and left and right translates of separable subsets are separable, so we may assume that $g_1=g_2=1$. 

Choose a basepoint $x\in X$ and write $G=\pi_1(X,x)$. By Lemma \ref{lem:convexisbased} there are based local isometries of finite cube complexes $(Y_i,y_i)\to (X,x)$ inducing the subgroups $K_i$. By Corollary \ref{cor:embedPi}, there is a finite directly special regular cover $(\hat{X},\hat{x})\to(X,x)$ such that the based elevations $(\hat{Y}_i,\hat{y}_i)\xhookrightarrow{}(\hat{X},\hat{x})$ of $(Y_i,y_i)\to (X,x)$ are embedded and do not inter-osculate with hyperplanes of $\hat{X}$. We will consider the $\hat{Y}_i$ as subcomplexes of $\hat{X}$ with $\hat{y}_i=\hat{x}$. If $\hat{G}\triangleleft G$ is the subgroup corresponding to $(\hat{X},\hat{x})\to(X,x)$ then $\hat{K}_i:=K_i\cap\hat{G}$ is the subgroup corresponding to $\hat{Y}_i\subset\hat{X}$. The triple coset $K_1 K_2 K_3$ is a finite union of triple cosets $g_1\hat{K}_1 \hat{K}_2 g_2\hat{K}_3 g_3$ with $g_i\in K_i$, so it is enough to prove that the triple cosets
$$\hat{K}_1 \hat{K}_2 g_2\hat{K}_3 g_2^{-1}$$
are separable for $g_2\in K_2$ (again using the fact that left and right translates of separable subsets are separable).

The subgroup $K'_3:=g_2\hat{K}_3 g_2^{-1}$ corresponds to the based elevation $(Y'_3,x')\xhookrightarrow{}(\hat{X},x')$, where $x'$ is the endpoint of the path $\hat{\gamma}_2$ in $\hat{Y}_2$ based at $\hat{x}$ which represents the element $g_2$ (see Remark \ref{remk:convexsubgp}). Since $\hat{X}\to X$ is regular, we know that $Y'_3$ is embedded in $\hat{X}$ and does not inter-osculate with any hyperplanes.

We now consider the imitator homomorphism $\rho_\phi:\hat{G}_{\phi}\to\hat{K}_2<\hat{G}_{\phi}$ of the inclusion $\phi:(\hat{Y}_2,\hat{x})\xhookrightarrow{}(\hat{X},\hat{x})$. Let $\dot{K}_1:=\hat{K}_1\cap\hat{G}_{\phi}$ and $\dot{K}'_3:=K'_3\cap\hat{G}_{\phi}$.
If the walker traverses a loop $\gamma$ in $\hat{Y}_1$ based at $\hat{x}$, then Subcomplex Entrapment implies that the imitator stays in $\hat{Y}_1\cap\hat{Y}_2$. It follows that $\rho_\phi(\dot{K}_1)<\dot{K}_1$. If the walker traverses a loop of the form $\hat{\delta}_2 * \gamma' * \hat{\delta}_2^{-1}$ that lies in the imitator subgroup $\hat{G}_{\phi}$, where $\gamma'$ is a loop in $Y'_3$ based at $x'$, then the imitator traverses a loop of the form $\hat{\delta}_2 * \gamma'' * \hat{\delta}_2^{-1}$, where $\gamma''$ lies in $\hat{Y}_2\cap Y'_3$ by Subcomplex Entrapment. It follows that $\rho_\phi(\dot{K}'_3)<\dot{K}'_3$. We can then apply Lemma \ref{lem:separabletcoset} below to deduce that $\dot{K}_1\hat{K}_2\dot{K}'_3$ is separable in $\hat{G}_{\phi}$ (note that $\dot{K}_1$-$\dot{K}'_3$ double cosets are separable by Proposition \ref{prop:doublecoset}). Since $\hat{G}_\phi$ has finite index in $\hat{G}$, it follows that the triple coset $\hat{K}_1 \hat{K}_2 g_2\hat{K}_3 g_2^{-1}$ is separable.
\end{proof}

\begin{lem}\label{lem:separabletcoset}
	Let $\rho:G\twoheadrightarrow K_2< G$ be a retraction of a group $G$, and let $K_1,K_3<G$ be subgroups whose double cosets are separable and satisfy $\rho(K_i)< K_i$. Then the triple coset $K_1 K_2 K_3$ is separable in $G$.
\end{lem}
\begin{proof}
	Let $g\in G- K_1 K_2 K_3$. The goal is to find a finite quotient of $G$ that separates $g$ from $K_1 K_2 K_3$. We know that $\rho(g)\in K_2$, so $g\notin K_1\rho(g)K_3$, and by separability of $K_1\rho(g)K_3$ there exists a finite quotient $q:G\to\bar{G}$ with
	\begin{equation}\label{qequationt}
		q(g)\notin q(K_1)\rho(g)q(K_3).
	\end{equation} 	
	Now consider another finite quotient of $G$, defined by
	\begin{align*}
		t:G&\to\bar{G}\times\bar{G}\\
		h&\mapsto(q(h),q\rho(h)).
	\end{align*}
	We claim that $t(g)\notin t(K_1 K_2 K_3)$, so suppose for contradiction that $t(g)=t(k_1k_2k_3)$ with $k_i\in K_i$. Then
	\begin{equation}\label{texpandt}
		(q(g),q\rho(g))=(q(k_1k_2k_3),q\rho(k_1k_2k_3)).
	\end{equation}
	But $\rho(k_2)=k_2$ and $\rho(k_i)\in K_i$ for $i=1,3$, so 
	\begin{align*}
q(k_1k_2k_3)&=q(k_1)(q\rho(k_1))^{-1} (q\rho(k_1))q(k_2)(q\rho(k_3)) (q\rho(k_3))^{-1}q(k_3)\\
&\in q(K_1) (q\rho(k_1k_2k_3))q(K_3),
	\end{align*}
 contradicting (\ref{qequationt}) and (\ref{texpandt}).
\end{proof}

\bigskip
\section{Trivial wall projections}

In this section we prove Theorem \ref{thm:trivialwallproj1} about elevations of subcomplexes with trivial wall projections. First we must recall the notion of projection maps in CAT(0) cube complexes.

\subsection{Projection maps}

\begin{defn}\label{defn:projection}(Projection maps)\\
	Let $X$ be a CAT(0) cube complex and let $Y\subset X$ be a convex subcomplex. The \emph{projection to $Y$} is the combinatorial map $\Pi:X\to Y$ that sends each vertex $x\in X^0$ to the unique closest vertex in $Y$ with respect to the combinatorial metric -- this is well-defined by \cite[Lemma 13.8]{HaglundWise08}.
\end{defn}

These projection maps can be used to construct ``bridges'' between convex subcomplexes as follows. This theorem appears in lecture notes of Hagen as \cite[Theorem 1.22]{Hagen19} (see also \cite{ChatterjiFernosIozzi16}). There are analogues of this theorem in the more general contexts of normal binary spaces \cite[Theorems 2.5 and 2.6]{Vandevel83} and gated sets \cite{DressScharlau87}, and a similar phenomenon also occurs for general CAT(0) spaces \cite[II.2.12(2)]{BridsonHaefliger99}.

\begin{thm}(Bridge Theorem)\\\label{thm:bridge}
Let $X$ be a CAT(0) cube complex and let $Y_1,Y_2\subset X$ be convex subcomplexes. Let $\Pi_i:X\to Y_i$ be the projection to $Y_i$.  Then:
\begin{enumerate}
	\item\label{item:bridgecross} A hyperplane $H$ crosses $\Pi_1(Y_2)$ if and only if $H$ crosses $Y_1$ and $Y_2$. A hyperplane $H$ separates $Y_1$, $Y_2$ if and only if $H$ separates $\Pi_1(Y_2)$, $\Pi_2(Y_1)$.
	\item\label{item:bridgeproduct} There is a convex product subcomplex $A\times B\subset X$ and vertices $a_1,a_2\in A$ such that $a_1\times B=\Pi_1(Y_2)$ and $a_2\times B=\Pi_2(Y_1)$.
\end{enumerate}
\end{thm}

\begin{prop}\label{prop:finiteprojection}
	Let $G\acts X$ be a cubulated group with free $G$-action, and let $K_1,K_2<G$ be convex subgroups that act cocompactly on convex subcomplexes $Y_1,Y_2\subset X$ respectively. Let $\Pi_i:X\to Y_i$ be the projection to $Y_i$. Then:
	\begin{enumerate}
		\item\label{item:finiteproj} The projection $\Pi_1(Y_2)$ is finite if and only if $K_1\cap K_2=\{1\}$.
		\item\label{item:embeddedproj} If $X/G$ is directly special and $Y_2/K_2$ embeds in $X/G$, then $K_1\cap K_2=\{1\}$ implies that $\Pi_1(Y_2)$ embeds in $Y_1/K_1$.
	\end{enumerate}
\end{prop}
\begin{proof}
	Suppose $\Pi_1(Y_2)$ is finite. $K_1\cap K_2$ stabilises $Y_1$ and $Y_2$, so it must also stabilise $\Pi_1(Y_2)$. Since $K_1\cap K_2$ acts freely on $\Pi_1(Y_2)$ we deduce that $K_1\cap K_2$ is finite. But $G$ is torsion free (because it acts freely on a CAT(0) space), so $K_1\cap K_2=\{1\}$.
	
	Conversely, suppose $\Pi_1(Y_2)$ is infinite. Let $x\in\Pi_1(Y_2)$, and take an infinite sequence of distinct elements $g_n\in K_1$ with $g_nx\in\Pi_1(Y_2)$. It follows from Theorem \ref{thm:bridge} that $\Pi_1(Y_2)$ and $\Pi_2(Y_1)$ lie within bounded neighbourhoods of each other, and so by the cocompactness of the $K_2$-action on $Y_2$ we can pick elements $h_n\in K_2$ such that $d(g_nx,h_nx)$ is uniformly bounded. Therefore the points $h_n^{-1}g_nx$ lie in a bounded neighbourhood of $x$, and by local finiteness of $X$ there exist $n\neq m$ with $h_n^{-1}g_nx=h_m^{-1}g_mx$. Thus $1\neq g_n g_m^{-1}=h_n h_m^{-1}\in K_1\cap K_2$. This completes the proof of \ref{item:finiteproj}.
	
	We now prove \ref{item:embeddedproj}, so suppose that $X/G$ is directly special,
	$Y_2/K_2$ embeds in $X/G$ and $\Pi_1(Y_2)$ doesn't embed in $Y_1/K_1$. Hence there exist vertices $x_1,y_1\in \Pi_1(Y_2)$ and $g\in K_1$ with $gx_1=y_1$. By Theorem \ref{thm:bridge}\ref{item:bridgeproduct}, there is a convex product subcomplex $A\times B\subset X$ and vertices $a_1,a_2\in A$ such that $a_1\times B=\Pi_1(Y_2)$ and $a_2\times B=\Pi_2(Y_1)$.
	 Suppose $x_1=(a_1,b)$ and $y_1=(a_1,c)$. Let $\gamma$ be a combinatorial geodesic in $A$ from $a_1$ to $a_2$. Then $\gamma\times\{b\}$ and $\gamma\times\{c\}$ are combinatorial geodesics in $X$ starting at $x_1$ and $y_1$ respectively, which cross the same sequence of hyperplanes. Their images in $X/G$ are paths with the same starting point that again cross the same sequence of hyperplanes. Since hyperplanes in $X/G$ do not self-intersect or self-osculate, we deduce that $\gamma\times\{b\}$ and $\gamma\times\{c\}$ project to the same path in $X/G$. Then $g(\gamma\times\{b\})=\gamma\times\{c\}$, and in particular $g(a_2,b)=(a_2,c)$. But $(a_2,b),(a_2,c)\in Y_2$ are vertices that map to the same vertex in $X/G$, hence $g\in K_2$ and $K_1\cap K_2\neq\{1\}$.
\end{proof}

\begin{remk}
	In Proposition \ref{prop:finiteprojection} \ref{item:embeddedproj} we didn't use the full strength of direct specialness of $X/G$, we only used the fact that hyperplanes don't self-intersect or self-osculate (sometimes known as weak specialness).
\end{remk}

\subsection{Wall projections}

We now recall the notion of wall projection, due to Haglund--Wise \cite[Definition 3.14]{HaglundWise12}, which will play a key role later in limiting the movement of imitators.

\begin{defn}(Wall projection)\\\label{defn:wallprojection}
	Let $X$ be a cube complex with subcomplexes $Y_1$ and $Y_2$. We define $\WProj_X(Y_1\to Y_2)$, the \emph{wall projection of $Y_1$ onto $Y_2$}, to equal the union of $Y_2^0$ together with all cubes of $Y_2$ whose edges are parallel to edges of $Y_1$. We say the wall projection is \emph{trivial} if any closed loop of $\WProj_X(Y_1\to Y_2)$ is homotopically trivial inside $X$.
\end{defn}

\begin{remk}\label{remk:wallprojNPC}
	If $X$ is non-positively curved and $Y_2$ is locally convex then $\WProj_X(Y_1\to Y_2)$ is also locally convex. So in this setting the wall projection $\WProj_X(Y_1\to Y_2)$ is trivial if and only if its components are simply connected.
\end{remk}

\begin{remk}\label{remk:wallprojimitator}
	Consider the walker-imitator construction for the inclusion $Y_2\xhookrightarrow{}X$. It follows immediately from the definitions that if the walker stays in $Y_1$ then the imitator stays in $\WProj_X(Y_1\to Y_2)$.
\end{remk}

\begin{remk}\label{remk:wallprojlift}
	If $\hat{Y}_1$ and $\hat{Y}_2$ are elevations of $Y_1$ and $Y_2$ to a finite cover $\mu:\hat{X}\to X$, then $\WProj_{\hat{X}}(\hat{Y}_1\to\hat{Y}_2)$ is contained inside the union of the elevations of $\WProj_X(Y_1\to Y_2)$ to $\hat{Y}_2$. This is because edges $e_i\in\hat{Y}_i^1$ being parallel implies that $\mu(e_i)\in Y_i^1$ are parallel. In particular, $\WProj_{\hat{X}}(\hat{Y}_1\to\hat{Y}_2)$ is trivial if $\WProj_X(Y_1\to Y_2)$ is trivial. 
\end{remk}

\subsection{Proof of Theorem \ref{thm:trivialwallproj1}}\label{subsec:thmtrivialwallproj1}

We now come to the main result of the section. This generalises \cite[Proposition 5.1]{HaglundWise12}, which is the equivalent statement for graphs. It should also be compared to \cite[Corollary 5.8]{HaglundWise12} (which was a key step in the proof of Haglund and Wise's combination theorem for special cube complexes); our result is less general in that it says nothing about wall projections of elevations of $Y_1$ onto other elevations of $Y_1$, but more general in that it does not assume hyperbolicity of $\pi_1 X$. 

\theoremstyle{plain}
\newtheorem*{thm:trivialwallproj1}{Theorem \ref{thm:trivialwallproj1}}
\begin{thm:trivialwallproj1}(Elevating to trivial wall projections)\\
	Let $Y_1,Y_2\to X$ be local isometries of finite virtually special cube complexes, and let $K_1,K_2<G:=\pi_1 X$ be the corresponding subgroups (well-defined up to conjugacy). Suppose that $K_1$ has trivial intersection with every conjugate of $K_2$. Then there is a finite directly special cover $\hat{X}\to X$ such that all elevations of $Y_1$ and $Y_2$ are embedded, and each elevation of $Y_1$ has trivial wall projection onto each elevation of $Y_2$.	
\end{thm:trivialwallproj1}

Before giving a precise proof, we sketch the main ideas involved, and compare them with the ideas from the Haglund--Wise result.
The main obstruction to $\WProj_X(Y_1\to Y_2)$ being trivial is having hyperplanes that go from $Y_1$ to $Y_2$ via many different routes, so we must peel away most of these hyperplanes when passing to any pair of elevations of $Y_1$ and $Y_2$ in the finite cover $\hat{X}$. The idea for our proof is to consider certain small simply connected subcomplexes $W_2\subset Y_2$, and certain pairs of hyperplanes $H,H'$ that cross both $W_2$ and $Y_1$, such that any loop $\gamma$ that passes through $W_2, H',Y_1,H$ in turn is essential (see Figure \ref{fig:compareproofs}). We then apply triple coset separability to the subgroups of $\pi_1 X$ corresponding to $H',Y_1,H$ to obtain a finite cover $\hat{X}$ of $X$ where such loops $\gamma$ lift to paths with distinct endpoints. This implies that no pair of elevations of $H$ and $H'$ to $\hat{X}$ cross a pair of elevations of $W_2$ and $Y_1$. We do this for all suitable $W_2,H,H'$, and the finite cover required by the theorem is any finite cover that factors through all of the finite covers we obtain.

The approach taken by Haglund--Wise in the proof of \cite[Corollary 5.8]{HaglundWise12} still involves peeling away hyperplanes, but does this in quite a different way, which relies on the hyperbolicity of $\pi_1 X$. Very roughly, if a hyperplane $H$ crosses $Y_1$ and $Y_2$, they use hyperbolicity of $\pi_1 X$ to obtain a locally convex subcomplex $Z$ that is a thickening of $Y_1, Y_2$ and $H$ (see Figure \ref{fig:compareproofs}). Similarly, for any other hyperplane $H'$ crossing $Y_1$ and $Y_2$ they obtain a thickening $Z'$. Then they pass to a finite cover $\hat{X}$ of $X$ where the elevations of $Z$ and $Z'$ (based at an elevation of $Y_1$) have connected intersection. If $H$ and $H'$ take different routes from $Y_1$ to $Y_2$, then they deduce that the elevations of $H$ and $H'$ (that are contained in the elevations of $Z$ and $Z'$) cross different elevations of $Y_2$.

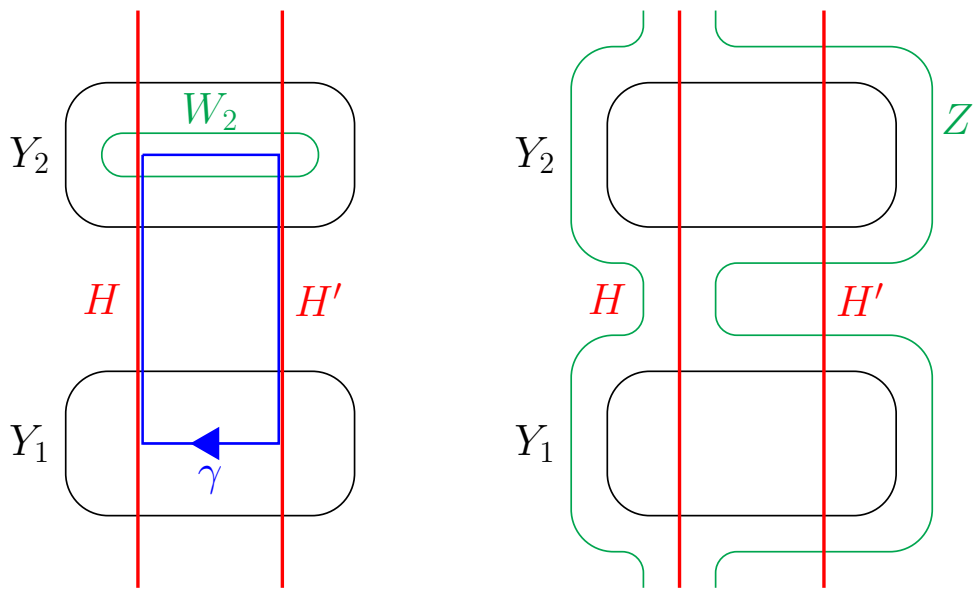
\begin{figure}[H]
	\centering
	\scalebox{0.8}{
\begin{tikzpicture}[auto,node distance=2cm,
	thick,every node/.style={circle,draw,font=\small},
	every loop/.style={min distance=2cm},
	hull/.style={draw=none},
	]
	\tikzstyle{label}=[draw=none,font=\huge]
	
	\begin{scope}[scale=.6]	
		\draw[rounded corners=20pt] (0, 0) rectangle (8, -4) {};
		\draw[rounded corners=20pt] (0, -8) rectangle (8, -12) {};
		\node[label] (Y2) at (-1,-2) {$Y_2$};	
		\node[label] (Y2) at (-1,-10) {$Y_1$};
		\coordinate (Ht) at (2,2);
		\coordinate (Hb) at (2,-14);
		\coordinate (H't) at (6,2);
		\coordinate (H'b) at (6,-14);
		
		\draw[rounded corners=10pt, Green] (1,-1.4) rectangle (7,-2.6) {};
		\node[label,Green] (W) at (4,-.8) {$W_2$};
		
		\draw[red,ultra thick] (Ht)--(Hb);
		\draw[red,ultra thick] (H't)--(H'b);
		\node[label,red] (H) at (1,-6) {$H$};
		\node[label,red] (H') at (7,-6) {$H'$};
		
		\draw[blue, very thick] (2.13,-2) -- (5.9,-2) -- (5.9,-10) --  (2.13,-10) -- (2.13,-2);
		\path (5.9,-10) edge [blue,postaction={decoration={markings,mark=at position 0.65 with {\arrow[blue,line width=1mm]{triangle 60}}},decorate}] (2.13,-10);
		\node[label,blue] (Ga) at (4,-11) {$\gamma$};
	\end{scope}
	
	\begin{scope}[shift={(9,0)},scale=.6]	
		\draw[rounded corners=20pt] (0, 0) rectangle (8, -4) {};
		\draw[rounded corners=20pt] (0, -8) rectangle (8, -12) {};
		\draw[rounded corners=10pt,Green] (1,2)--(1,1)[rounded corners=20pt]--(-1,1)--(-1,-5)[rounded corners=10pt]--(1,-5)--(1,-7)[rounded corners=20pt]--(-1,-7)
		--(-1,-13)[rounded corners=10pt]--(1,-13)--(1,-14);
		\draw[rounded corners=10pt,Green] (3,2)--(3,1)[rounded corners=20pt]--(9,1)--(9,-5)[rounded corners=10pt]--(3,-5)--(3,-7)[rounded corners=20pt]--(9,-7)
		--(9,-13)[rounded corners=10pt]--(3,-13)--(3,-14);
		\node[label,Green] (Z) at (9.7,-1) {$Z$};
		\node[label] (Y2) at (-2,-2) {$Y_2$};	
		\node[label] (Y2) at (-2,-10) {$Y_1$};
		\coordinate (Ht) at (2,2);
		\coordinate (Hb) at (2,-14);
		\coordinate (H't) at (6,2);
		\coordinate (H'b) at (6,-14);
		
		\draw[red,ultra thick] (Ht)--(Hb);
		\draw[red,ultra thick] (H't)--(H'b);
		\node[label,red] (H) at (0,-6) {$H$};
		\node[label,red] (H') at (7,-6) {$H'$};

	\end{scope}
	
\end{tikzpicture}
}
\caption{Idea for the proof of Theorem \ref{thm:trivialwallproj1}, left, idea for the proof of Haglund--Wise \cite[Corollary 5.8]{HaglundWise12}, right.}\label{fig:compareproofs}
\end{figure}

\begin{proof}[Proof of Theorem \ref{thm:trivialwallproj1}]
	We may assume that $Y_1$ and $Y_2$ are already embedded in $X$, as otherwise we can pass to a finite cover $X'\to X$ in which all their elevations are embedded, and then apply the proposition to each pair of elevations $Y'_1$ and $Y'_2$ of $Y_1$ and $Y_2$ in $X'$. We can then take $\hat{X}\to X$ to be a finite cover factoring through all these covers, and by Remark \ref{remk:wallprojlift} each elevation of $Y_1$ to $\hat{X}$ will have trivial wall projection onto each elevation of $Y_2$. Similarly, we may assume that $X$ is already directly special.
	
	Now consider the universal cover $\widetilde{X}\to X$, and let $\widetilde{Y}_2$ be an elevation of $Y_2$. Let $\Pi:\widetilde{X}\to\widetilde{Y}_2$ be the projection to $\widetilde{Y}_2$.	
	Proposition \ref{prop:finiteprojection}\ref{item:embeddedproj} (with Remark \ref{remk:convexsubgp}) implies that the elevations $\widetilde{Y}_1$ of $Y_1$  have projections $\Pi(\widetilde{Y}_1)$ whose diameters are bounded by some constant $L$ (with respect to the combinatorial metric say).
	 	
	Consider the $d_\infty$-metric on $\widetilde{X}^0$, which corresponds to the graph  obtained from the 1-skeleton $\widetilde{X}^1$ by adding an edge between two vertices whenever they belong to a common cube \cite[p1542]{Genevois21}. Equivalently, $d_\infty(\wt{x}_1,\wt{x}_2)$ is the maximal number of pairwise disjoint hyperplanes separating $\wt{x}_1$ and $\wt{x}_2$ \cite[Proposition 2.4]{Genevois21}.
	Pick a vertex $x\in Y_2$ and take a lift $\widetilde{x}\in\widetilde{Y}_2$.
	Let $W(\widetilde{x})$ be the induced subcomplex of the ball of radius $L+1$ in the $d_\infty$-metric about $\widetilde{x}$. 
	It is clear from the definition of $d_\infty$ that $W(\widetilde{x})$ is a finite subcomplex, and it follows from \cite[Corollary 3.5]{HruskaWise09} that $W(\wt{x})$ is convex.
	Let $\calB(\widetilde{x})$ be the finite collection of hyperplanes dual to edges $e$ that have exactly one endpoint in $W(\widetilde{x})$.\\
	
	\begin{claim}
		Every $H\in\calB(\widetilde{x})$ is separated from $\widetilde{x}$ by a collection of $L+1$ pairwise disjoint hyperplanes.
	\end{claim}\\
	
	\begin{claimproof}
		Let $H\in\calB(\widetilde{x})$ be dual to an edge $e$. Let $e$ have endpoints $\wt{y}_1$ and $\wt{y}_2$, with $\wt{y}_2\in W(\wt{x})$.
		We have $d_\infty(\wt{x},\wt{y}_1)\geq L+2$, so $\wt{x}$ is separated from $\wt{y}_1$ by a collection $\mathcal{F}$ of $L+2$ pairwise disjoint hyperplanes.
		The hyperplane in $\mathcal{F}$ nearest $\wt{y}_1$ must be $H$, otherwise $\wt{y}_2$ would also be separated from $\wt{x}$ by $\mathcal{F}$, contradicting $d_\infty(\wt{x},\wt{y}_2)\leq L+1$.
		Hence $H$ is separated from $\wt{x}$ by the $L+1$ pairwise disjoint hyperplanes in $\mathcal{F} - \{H\}$, as required.
\end{claimproof}\\	
	
	Now let $W_2(\wt{x}):=W(\wt{x})\cap\wt{Y}_2$ and let $\calB_2(\wt{x})$ consist of those hyperplanes in $\calB(\wt{x})$ that intersect $\wt{Y}_2$.
	For each $H\in\calB_2(\wt{x})$ pick a hyperplane $H'$ separating $H$ from $\wt{x}$ and such that $H$ and $H'$ are separated by a collection of $L$ pairwise disjoint hyperplanes (such $H'$ exist by the claim).
	Let $\Omega(\wt{x},H')$ be the collection of all elevations $\wt{Y}_1$ of $Y_1$ that intersect $H'$.
	For each $\wt{Y}_1\in\Omega(\wt{x},H')$ we know that $\Pi(\wt{Y}_1)$ has diameter at most $L$, so it follows from the construction of $H'$ that $\Pi(\tilde{Y}_1)$ is disjoint from $H$.
	Furthermore, we deduce from Theorem \ref{thm:bridge}\ref{item:bridgecross} that $\wt{Y}_1$ is also disjoint from $H$.
		
	The key step will be to use triple coset separability to prove the following claim.\\
	
	\begin{claim}
There exists a finite cover $\hat{X}\to X$ such that, for all $H\in\calB_2(\widetilde{x})$ and $\widetilde{Y}_1\in\Omega(\widetilde{x},H')$, their images in $\hat{X}$ are disjoint.
	\end{claim}\\

\begin{claimproof}
The property of the claim passes to further covers, so we can prove the claim by working with one hyperplane $H$ from the finite collection $\calB_2(\widetilde{x})$ at a time.
For $H\in\calB_2(\widetilde{x})$ let $G_{H'}$ be the stabiliser of $H'$.
The action of $G_{H'}$ on $H'$ is cocompact, so $\Omega(\wt{x},H')$ is a finite union of $G_{H'}$-orbits of the subcomplexes $\wt{Y}_1$.
It suffices to prove the claim for one such orbit $G_{H'}\cdot\wt{Y}_1$.
 
The finite cover $\hat{X}\to X$ will be defined by a finite-index normal subgroup $\hat{G}\triangleleft G$, so our task is to pick this $\hat{G}$ such that there do not exist $g\in \hat{G}$ and $h'\in G_{H'}$ with
\begin{equation}\label{badgh'}
	H\cap gh'\widetilde{Y}_1\neq\emptyset.
\end{equation}
Assume that $K_1$ is the $G$-stabiliser of our chosen elevation $\widetilde{Y}_1$, and write $G_H$ for the stabiliser of $H$. 
Working with the CAT(0) metric, suppose that any $R$-ball in $\widetilde{Y}_1$ (resp. $H$) has $K_1$-translates (resp. $G_H$-translates) that cover $\widetilde{Y}_1$ (resp. $H$). 
Pick a vertex $\wt{y}\in\wt{Y}_1$.
If $g\in G$ and $h'\in G_{H'}$ satisfy (\ref{badgh'}), then there exists $k\in K_1$ such that $d(H,gh'k\widetilde{y})\leq R$. Then there also exists $h\in G_H$ such that $d(\widetilde{x},hgh'k\widetilde{y})\leq 2R+d(\widetilde{x},H)$. Writing $g':=hgh'k$, we know that $g'$ also satisfies (\ref{badgh'}). 

All $G_{H'} K_1$-translates of $\wt{Y}_1$ are again in $\Omega(\wt{x},H')$, so are disjoint from $H$, hence no element of the triple coset $G_H G_{H'} K_1$ satisfies (\ref{badgh'}).
By Proposition \ref{prop:triplecoset}, $G_H G_{H'} K_1$ is separable in $G$, so there is a finite-index subgroup $\hat{G}<G$ with
\begin{equation}\label{separateg'2}
g'\hat{G}\cap G_H G_{H'} K_1=\emptyset.
\end{equation}
Replacing $\hat{G}$ by its normal core, we may assume that $\hat{G}$ is normal in $G$.
We have
$$g'\in G_H g G_{H'} K_1,$$
so we deduce that $g\notin\hat{G}$. There are only finitely many $g'\in G$ with $d(\widetilde{x},g'\widetilde{y})\leq 2R+d(\widetilde{x},H)$, so we can satisfy (\ref{separateg'2}) for all possible $g'$ that arise in this way by intersecting the corresponding subgroups $\hat{G}$.
This gives us a finite-index normal subgroup $\hat{G}\triangleleft G$ such that no $g\in \hat{G}$ and $h'\in G_{H'}$ satisfy (\ref{badgh'}), as required.
\end{claimproof}\\	 
 
 At this point we have coverings $\widetilde{X}\to\hat{X}\to X$; let $\hat{Y}_2$ be the image of $\widetilde{Y}_2$ in $\hat{X}$ and let $\hat{x}$ be the image of $\widetilde{x}$. Just as we assumed that $Y_1$ and $Y_2$ are already embedded in $X$, we may assume that the map $W_2(\widetilde{x})\to X$ is an embedding, hence the map $W_2(\widetilde{x})\to\hat{X}$ is also an embedding. Considering $W_2(\widetilde{x})$ as a subcomplex of $\hat{X}$, we know that $\hat{x}\in W_2(\widetilde{x})\subset \hat{Y}_2$. Let $\hat{Y}_1$ be an elevation of $Y_1$ to $\hat{X}$. We now make a third claim.\\
 
 \begin{claim}
 	The component of $\WProj_{\hat{X}}(\hat{Y}_1\to\hat{Y}_2)$ containing $\hat{x}$ lies inside $W_2(\widetilde{x})$.
 \end{claim}\\

\begin{claimproof}
Suppose not. Then there exists a path $\hat{\gamma}$ in $\hat{Y}_2$ based at $\hat{x}$ that leaves $W_2(\widetilde{x})$ and is such that every edge of $\hat{\gamma}$ is parallel to some edge of $\hat{Y}_1$. 
Say that the edges of $\hat{\gamma}$ are dual to hyperplanes $\hat{H}_1,\hat{H}_2,...,\hat{H}_n$ respectively.
So all $\hat{H}_i$ intersect both $\hat{Y}_1$ and $\hat{Y}_2$.
Lift $\hat{\gamma}$ to a path $\gamma$ in $\widetilde{X}$ based at $\widetilde{x}$, and suppose that its edges are dual to hyperplanes $H_1,H_2,...,H_n$ respectively.
Since $\hat{\gamma}$ leaves $W_2(\wt{x})$, we must have $H_i\in\calB(\wt{x})$ for some $i$, and $H'_i=H_j$ for some $j<i$.
Choose an elevation $\wt{Y}_1$ of $\hat{Y}_1$ that intersects $H_j$.
Then $\wt{Y}_1\in\Omega(\wt{x},H_j)$, but the images of $\wt{Y}_1$ and $H_i$ in $\hat{X}$ are $\hat{Y}_1$ and $\hat{H}_i$, which intersect. Thus we contradict the previous claim.
\end{claimproof}\\

$W_2(\widetilde{x})$ is a convex subcomplex of $\widetilde{X}$, hence simply connected, so any loop in $\WProj_{\hat{X}}(\hat{Y}_1\to\hat{Y}_2)$ based at $\hat{x}$ is homotopically trivial inside $\hat{X}$. As $\hat{X}\to X$ is a regular cover, the same is true for any elevations $\hat{Y}_1$ and $\hat{Y}_2$ of $Y_1$ and $Y_2$, and any lift $\hat{x}$ of $x$ contained inside $\hat{Y}_2$. Running the entire argument for each vertex $x\in Y_2$ gives us a finite collection of finite covers of $X$, and any finite cover $\hat{X}\to X$ factoring through all of them will satisfy the conclusions of the theorem.
\end{proof}

\bigskip
\section{Commanding elements and subgroups}\label{sec:commanding}

In this section we step away from cubical geometry to consider the group theoretic notions of commanding elements and subgroups.
These definitions are motivated by existing theorems in the literature, most of which are designed as tools for building finite covers of graphs of groups. In particular, the notion of commanding elements is inspired by that of omnipotence (as discussed in the introduction), and we will see that Wise's Malnormal Special Quotient Theorem gives rise to many examples of commanding subgroups. We also give a general proposition about how commanding subgroups can be used to deduce profinite and separability properties for graphs of groups.
We close the section with a technical lemma, needed later in the paper, about commanding elements and passing to finite-index subgroups.

\subsection{Definitions and examples}\label{subsec:commanddefn}

\begin{defn}(Commanding group elements)\\
	A group $G$ \emph{commands} a set of elements $\{g_1,...,g_n\}\subset G$ if there exists an integer $N>0$ such that for any integers $r_1,...,r_n>0$ there exists a homomorphism to a finite group $G\to\bar{G},g\mapsto\bar{g}$ such that the order of $\bar{g}_i$ is $Nr_i$. If this can always be done with $\langle\bar{g}_i\rangle\cap\langle\bar{g}_j\rangle=\{1\}$ for all $i\neq j$ then we say that $G$ \emph{strongly commands} $\{g_1,...,g_n\}$.
\end{defn}

There is a natural generalisation to commanding subgroups, as follows.

\begin{defn}(Commanding subgroups)\\\label{defn:commandingsubgroups}
	A group $G$ \emph{commands} a collection of subgroups $(P_1,...,P_n)$ if there exist finite-index subgroups $\dot{P}_i<P_i$ such that, for any choice of finite-index subgroups $P'_i<\dot{P}_i$ with $P'_i\triangleleft P_i$, there exists a finite-index normal subgroup $G'\triangleleft G$ such that $P_i\cap G'=P'_i$ for $1\leq i\leq n$.
	If this can always be done with $P_i G'\cap P_j G'=G'$ for any $i\neq j$ then we say that $G$ \emph{strongly commands} $(P_1,...,P_n)$.
	Note that the collection of subgroups $(P_1,...,P_n)$ can contain duplicates; but if $G$ strongly commands them then any duplicates must be trivial subgroups; or if $G$ is residually finite and commands $(P_1,...,P_n)$ then any duplicates must be finite subgroups.
\end{defn}

\begin{remk}
	$G$ (strongly) commands a set of infinite order elements $\{g_1,...,g_n\}$ if and only if it (strongly) commands the subgroups they generate $(\langle g_1\rangle,...,\langle g_n\rangle)$.
\end{remk}

\begin{remk}\label{remk:commandconjugates}
	$G$ commands $(P_1,...,P_n)$ implies that $G$ commands $(P_1,...,gP_i g^{-1},...,P_n)$ for any $g\in G$ and $1\leq i\leq n$.
\end{remk}

The following is an easy example of commanding subgroups.

\begin{prop}\label{prop:abeliancommand}
	Let $A$ be a finitely generated free abelian group. Then $A$ commands a collection of subgroups $(A_1,...,A_n)$ if and only if they are linearly independent (i.e. $\sum_i a_i=0$ for $a_i\in A_i$ implies $a_i=0$ for all $i$). Moreover, in this case $A$ will strongly command them.
\end{prop}
\begin{proof}
	Suppose $(A_1,...,A_n)$ are linearly independent. We will show that $A$ strongly commands $(A_i)$ with $\dot{A}_i=A_i$. Let $A'_i<A_i$ be finite-index subgroups. Put $B:=A'_1+...+A'_n$ and let $A/B=\mathbb{Z}^k\oplus F$ with $F$ finite. Let $C<A$ be a subgroup that intersects $B$ trivially and surjects onto the $\mathbb{Z}^k$ factor of $A/B$. Then $A':=B+C$ has finite index in $A$, and $A_i\cap A'=A'_i$ by the linear independence of $(A_i)$. Moreover, for $j\neq k$ we have
	\begin{align*}
		(A'+A_j)\cap(A'+A_k)&=(B+A_j)\cap(B+A_k)+C\\
		&=B+C\\
		&=A',
	\end{align*}
by linear independence of $(A_i)$.

Conversely, suppose $(A_i)$ are not linearly independent. Then there exist $a_i\in A_i$ not all zero with $\sum_i a_i=0$. Let $\dot{A}_i< A_i$ be finite-index subgroups. We may assume that $a_i\in\dot{A}_i$ by multiplying the $a_i$ by a large integer (remember that $A$ is torsion-free). Let $a_j\neq 0$, and set $A'_i=\dot{A}_i$ for $i\neq j$ and $A'_j<\dot{A}_j$ some finite-index subgroup with $a_j\notin A'_j$. Then there is no finite-index subgroup $A'<A$ with $A_i\cap A'=A'_i$, because $A'$ would contain the $a_i$ for $i\neq j$, and so would also contain $a_j=-\sum_{i\neq j}a_i$ which is in $A_j-A'_j$.
\end{proof}

The following deep theorem is a consequence of Wise's Malnormal Special Quotient Theorem. This was explained to the author by Daniel Woodhouse.

\begin{defn}(Almost malnormal)\\
	A collection of subgroups $\{P_1,...,P_n\}$ in a group $G$ is \emph{almost malnormal} if $gP_ig^{-1}\cap P_j $ is finite whenever $g\notin P_i$ or $i\neq j$.
\end{defn}

\theoremstyle{plain}
\newtheorem*{thm:commandqc}{Theorem \ref{thm:commandqc}}
\begin{thm:commandqc}
	Every virtually special hyperbolic group commands every almost malnormal collection of quasiconvex subgroups.
\end{thm:commandqc}
\begin{proof}
	Let $G$ be a virtually special hyperbolic group, and let $\{P_1,...,P_n\}$ be an almost malnormal collection of quasiconvex subgroups. 
	Then by Wise's Malnormal Special Quotient Theorem \cite[Theorem 12.3]{WiseQCH}, there are finite-index subgroups $\dot{P}_i<P_i$ such that, for any choice of finite-index subgroups $P'_i<\dot{P}_i$ with $P'_i\triangleleft P_i$, the quotient
	$$\bar{G}:=G/\llangle P'_1,...,P'_n\rrangle$$ 
	is virtually special and hyperbolic.
	By Remark \ref{remk:torsionfree}, $\bar{G}$ contains a finite-index torsion-free subgroup $\bar{G}'\triangleleft \bar{G}$. Let $G'\triangleleft G$ be the preimage of $\bar{G}'$ in $G$. Each subgroup $P_i$ will have finite image in $\bar{G}$, so this image will intersect $\bar{G}'$ trivially, thus $P_i\cap G'=\ker(P_i\to\bar{G})$.
	$G$ is hyperbolic relative to $\{P_1,...,P_n\}$ by \cite[Theorem 7.11]{Bowditch12}, so we may apply \cite[Theorem 1.1(1)]{Osin07}. This gives us a finite set of non-trivial elements $\mathcal{F}\subset G$ (which only depends on the $P_i$), such that $\ker(P_i\to\bar{G})=P'_i$ for each $i$ if $P'_i\cap\mathcal{F}=\emptyset$ for each $i$. So we are done if we replace the $\dot{P}_i$ by smaller subgroups that miss the finite set $\mathcal{F}$, which is possible because $G$ is residually finite.
\end{proof}

We obtain an even more general theorem if we instead use Einstein's relatively hyperbolic version of the Malnormal Special Quotient Theorem \cite[Theorem 2]{Einstein19} in the above proof (again this was pointed out by Woodhouse).

\begin{thm}\label{thm:commandrel}
	Let $G$ be a virtually special group that is hyperbolic relative to subgroups $\{P_1,...,P_n\}$. Then $G$ commands $(P_1,...,P_n)$.
\end{thm}

\begin{remk}\label{remk:relhypisconvex}
In Theorem \ref{thm:commandrel}, if $G$ admits a cubulation $G\acts X$, then the subgroups $P_i$ are necessarily convex \cite[Remark 1.3]{GrovesManning20}. Also, the assumption that $G\acts X$ is virtually special can be reduced to certain assumptions about how the subgroups $P_i$ interact with $X$ \cite{Oregonreyes20,GrovesManning20}.
\end{remk}

The following conjecture is a natural extension of Theorem \ref{thm:commandrel}. We prove some weaker versions of this in Section \ref{sec:buildingcommanding} (Theorems \ref{thm:nonregc} and \ref{thm:weakcommandqc}).

\theoremstyle{plain}
\newtheorem*{con:command1}{Conjecture \ref{con:command1}}
\begin{con:command1}
	Every virtually special cubulated group $G\acts X$ commands every almost malnormal collection of convex subgroups.
\end{con:command1}

If all hyperbolic groups are virtually torsion-free (or equivalently if all hyperbolic groups are residually finite \cite[Theorem 5.1]{KapovichWise00}) then we can extend Theorem \ref{thm:commandqc} to all hyperbolic groups (again this was pointed out by Woodhouse). 

\begin{thm}\label{thm:hypRF}
	If all hyperbolic groups are virtually torsion-free, then every hyperbolic group commands every almost malnormal collection of quasiconvex subgroups.
\end{thm}
\begin{proof}
	Run the same argument as in the proof of Theorem \ref{thm:commandqc}, but using Osin's peripheral filling theorem \cite[Theorem 1.1(1)]{Osin07} instead of Wise's Malnormal Special Quotient Theorem. This only gives us that the quotient $\bar{G}$ is hyperbolic, but by assumption this implies that $\bar{G}$ is virtually torsion-free, which is all that's needed to complete the argument.
\end{proof}

\subsection{Application to graphs of groups}\label{subsec:GoG}

Commanding subgroups is a very useful property when trying to build finite covers of graphs of groups (or graphs of spaces), which in turn can be used to prove separability or commensurability properties for groups. This idea has occurred several times before in the literature, even though the terminology of commanding is new to this paper.
We already discussed in the introduction how commanding elements is related to omnipotence, and Wise actually defined omnipotence in order to characterise when graphs of free groups with cyclic edge groups are subgroup separable \cite{Wise00}.
In a far more general setting, Wise used his Malnormal Special Quotient Theorem in order to prove his Quasi-convex Hierarchy Theorem about graphs of hyperbolic virtually special groups \cite{WiseQCH}, albeit this didn't go via any analogue of Theorem \ref{thm:commandqc}.
In addition, Wise effectively proved that finite-volume hyperbolic 3-manifold groups command their cusp subgroups \cite[Corollary 16.15]{WiseQCH} (note that Wise forgot to specify the dimension), and this was subsequently used by Behrstock and Neumann to prove a quasi-isometric rigidity result for certain non-geometric 3-manifold groups \cite{BehrstockNeumann12} - in this case the graph of spaces comes from the geometric decomposition of the 3-manifold.
Also, the author used Theorem \ref{thm:commandqc} in his proof of Agol's theorem on hyperbolic cubulated groups \cite{Shepherd19a}, and he used Theorem \ref{thm:commandrel} in joint work with Woodhouse to prove a quasi-isometric rigidity result for graphs of virtually free groups \cite{ShepherdWoodhouse20}.

The above papers always used commanding (or some version of it) in conjunction with extra properties about the specific groups involved.
In Proposition \ref{prop:GoG} below we give a very general statement about how commanding subgroups can be used to deduce profinite and separability properties for graphs of groups. First we fix some notation for graphs of groups -- see \cite{SerreTrees, ScottWall79, Bass93} for complete definitions and further background. 

\begin{nota}
	Let $G$ split as a graph of groups over a graph $\Gamma$. We will refer to the vertex and edge groups as $G_v$ and $G_e$, with $v\in V\Gamma$ a vertex and $e\in E\Gamma$ an oriented edge. We will also consider the vertex and edge groups as subgroups of $G$ (which is well-defined up to conjugation). For $e\in E\Gamma$ let $\bar{e}$ denote the same edge with reversed orientation and let $\tau(e)$ denote the terminal vertex of $e$. We make the identification $G_e=G_{\bar{e}}$. We denote the edge morphisms by $\zeta_e:G_e\to G_{\tau(e)}$. Given a vertex group $G_v$, the collection $(\zeta_e(G_e)\mid \tau(e)=v)$ will be referred to as its \emph{collection of incident edge groups}.
\end{nota}

\begin{prop}\label{prop:GoG}
Let $G$ split as a graph of groups over a finite graph $\Gamma$, and suppose that each vertex group commands its collection of incident edge groups. Then we have the following:
\begin{enumerate}
	\item\label{item:equaltopologies} For any vertex group $G_v<G$, the profinite topology on $G_v$ coincides with the subspace topology on $G_v$ induced by the profinite topology on $G$.
	\item\label{item:separabletoseparable} If every incident edge group is separable in its corresponding vertex group, then the vertex groups are separable in $G$.
	\item\label{item:resfinite} If every vertex group is residually finite and every incident edge group is separable in its corresponding vertex group, then $G$ is residually finite.
\end{enumerate}
\end{prop}

\begin{remk}
Before giving the proof, we remark that the properties appearing in Proposition \ref{prop:GoG} have been well studied in contexts other than commanding.
Tavgen' and Zalesski\u{\i} considered properties \ref{item:equaltopologies} and \ref{item:separabletoseparable} when proving conjugacy separability for certain amalgams of polycyclic groups \cite{TavgenZalesskii95}.
Hempel used similar ideas to prove that fundamental groups of Haken 3-manifolds are residually finite \cite{Hempel84}.
In \cite{RibesZalesskii10}, a graph of groups is called \emph{efficient} precisely if the conclusions hold for properties \ref{item:equaltopologies}--\ref{item:resfinite}, and this is needed if one wants to build the profinite completion of a graph of groups by piecing together the profinite completions of its vertex groups.
\end{remk}

\begin{proof}[Proof of Proposition \ref{prop:GoG}]
The vertex groups command their incident edge groups, so by definition there are finite-index subgroups $\dot{G}_e\triangleleft G_e$ for $e\in E\Gamma$ such that, for any choice of finite-index subgroups $G'_e<\dot{G}_e$ with $G'_e\triangleleft G_e$, there exist finite-index normal subgroups $G'_v\triangleleft G_v$ for $v\in V\Gamma$ with 
\begin{equation}\label{zetaeq}
\zeta_e(G_e)\cap G'_{\tau(e)}=\zeta_e(G'_e) 
\end{equation} 
for $e\in E\Gamma$.
Moreover, given finite-index normal subgroups $\hat{G}_v\triangleleft G_v$, if we choose the $G'_e$ such that $\zeta_e(G'_e)<\hat{G}_{\tau(e)}$, then we may assume that $G'_v<\hat{G}_v$ for all $v\in V\Gamma$ (by replacing the $G'_v$ with $G'_v\cap \hat{G}_v$ if necessary). Furthermore, if $G'_e=G'_{\bar{e}}$ for all $e\in E\Gamma$, then we can construct a finite cover of the graph of groups, with copies of the $G'_v$ and $G'_e$ as vertex and edge groups, and with edge morphisms obtained by composing the existing edge morphisms $\zeta_e$ with inner automorphisms of the original vertex groups $G_v$. This gives a valid cover of graphs of groups precisely because of equation (\ref{zetaeq}). 
We are now ready to prove \ref{item:equaltopologies} and \ref{item:separabletoseparable}.

\begin{enumerate}
	\item Fix a vertex group $G_v$.
	The cosets of finite-index normal subgroups of $G_v$ form a basis for the profinite topology on $G_v$, whereas the subspace topology on $G_v$ has basis consisting of cosets of intersections of $G_v$ with finite-index subgroups of $G$.
	It is clear from this description that open sets in the subspace topology are open in the profinite topology. To prove the converse we must show that any finite-index normal subgroup $\hat{G}_v\triangleleft G_v$ is open in the subspace topology.
	Indeed, given such $\hat{G}_v$, take a finite-index subgroup $G'<G$ corresponding to a finite cover of graphs of groups as above, such that $G'_v<\hat{G}_v$. Observe that $G'\cap G_v=G'_v$ is closed in the subspace topology on $G_v$. But $G'_v$ has finite index in $G_v$, so it is also open in the subspace topology. Hence $\hat{G}_v$ is open in the subspace topology too.  	
	\item Now suppose the vertex groups are residually finite. Fix a vertex group $G_v$ and an element $g\in G-G_v$. We must find a finite-index subgroup $G'<G$ with $g\notin G'G_v$ (as then $G'g$ is an open neighbourhood of $g$ distinct from $G_v$). 
	Now let $(X,\Gamma)$ be a graph of spaces corresponding to $(G,\Gamma)$ (see e.g. \cite{ScottWall79}), and let $\gamma$ be a loop in $X$ corresponding to $g$. We may assume that the basepoint is in the vertex space $X_v$. Suppose that $\gamma$ projects to an edge loop $\delta=(e_1,...,e_n)$ in $\Gamma$ with vertex sequence $v=v_0,v_1,...,v_n=v$.	
	After homotoping, we may assume that $\gamma$ is a concatenation of standard paths that go across the edge spaces $X_{e_i}$, together with based loops in the vertex spaces corresponding to elements $g_i\in G_{v_i}$. We may also assume that $\gamma$ is \emph{reduced} in the sense that $g_i\notin \zeta_{e_i}(G_{e_i})$ whenever $e_i=\bar{e}_{i+1}$. 	
	
	 By separability of the incident edge groups, there exist finite-index normal subgroups $\hat{G}_v\triangleleft G_v$ such that
	\begin{equation}\label{nobacktrack}
g_i\notin \zeta_{e_i}(G_{e_i})\hat{G}_{v_i},
	\end{equation}  
whenever $e_i=\bar{e}_{i+1}$. Now construct a finite cover $(G',\Gamma')$ of the graph of groups $(G,\Gamma)$ as described above, with $G'_v<\hat{G}_v$ for all $v\in V\Gamma$, and let $X'\to X$ be the corresponding finite cover of graphs of spaces. 
Let $\gamma$ lift to a based path $\gamma'$ in $X'$, and suppose this projects to a path $\delta'$ in $\Gamma'$.
It follows from (\ref{nobacktrack}) that $\delta'$ never backtracks, and if we construct $(G',\Gamma')$ appropriately we can ensure that $\delta'$ is embedded. The subgroup $G'G_v$ corresponds to based paths in $X'$ that start and finish at the same vertex space, so we conclude that $g\notin G'G_v$ as required (note that $n\geq1$ since $g\notin G_v$).
\item For any vertex group $G_v$, \ref{item:separabletoseparable} implies that $G_v$ is closed in the profinite topology on $G$. Then it follows from \ref{item:equaltopologies} and the residual finiteness of $G_v$ that the trivial subgroup is also closed in $G$, hence $G$ is residually finite.
\end{enumerate}
\end{proof}

\subsection{Passing to finite-index subgroups}

When we prove Theorem \ref{thm:command1} in Section \ref{sec:buildingcommanding} we will first exhibit a certain finite-index subgroup $\dot{G}\triangleleft G$ with a certain commanding-like property, and we will need the following lemma to transfer back to the whole of $G$.

\begin{lem}\label{lem:commandfromsubgp}
	Let $G$ be a finitely generated group with independent elements $\{g_1,...,g_n\}$. Suppose that $\dot{G}\triangleleft G$ is a finite-index normal subgroup and suppose that $M_i$ is the order of $g_i$ in the quotient $G/\dot{G}$.
	The $G$-conjugacy class of $h_i:=g_i^{M_i}\in\dot{G}$ splits into a finite number of $\dot{G}$-conjugacy classes, and we let $\{h_{ik}\}\subset\dot{G}$ be a set of representatives (with one $h_{ik}$ equal to $h_i$). Suppose for each $i$ there is a subcollection $\{h_{ik}\mid k\in K_i\}$, which includes $h_i$, and suppose there are integers $\dot{N}_{ik}>0$ for $k\in K_i$ with the following property: For any integers $\dot{r}_i>0$ there exists a finite quotient $\dot{G}\to\dot{G}/\dot{G}'$, where $h_{ik}$ has order $\dot{N}_{ik}\dot{r}_i$ for $i=1,...,n$ and $k\in K_i$, and $h_{ik}\in\dot{G}'$ otherwise. Then $G$ commands $\{g_i\}$.
	
	If in addition the finite quotients $\dot{G}\to\dot{G}/\dot{G}'$ described above can be chosen such that the images of the subgroups $\langle h_i\rangle$ have pairwise trivial intersections, then $G$ strongly commands $\{g_i\}$.
\end{lem}
\begin{proof}
	Put
	$$N:=\prod_{i=1}^n M_i\prod_{k\in K_i} \dot{N}_{ik}.$$
	Let $r_i>0$ be integers for $i=1,...,n$. In order to show that $G$ commands $\{g_i\}$, our task is to find a finite quotient of $G$ where $g_i$ has order $Nr_i$. Put
	$$\dot{r}_i:=\frac{Nr_i}{M_i \LCM\{\dot{N}_{ik}\mid k\in K_i\}},$$
	where $\LCM$ denotes the lowest common multiple of a set of integers. By the hypothesis of the lemma, we can now take a finite quotient $\dot{G}\to\dot{G}/\dot{G}'$, where $h_{ik}$ has order $\dot{N}_{ik}\dot{r}_i$ for $i=1,...,n$ and $k\in K_i$, and $h_{ik}\in\dot{G}'$ otherwise. Our finite quotient of $G$ will be defined using the following subgroup:
	$$G':=\bigcap_{g\in G} g\dot{G}' g^{-1}$$
	This is a finite-index normal subgroup of $G$ because $\dot{G}'<G$ has finite index and $G$ is finitely generated.
	Denote the quotient map $G\to G/G'$ by $g\mapsto\bar{g}$.\\
	
	\begin{claim}
		$\bar{g}_i$ has order $Nr_i$.
	\end{claim}\\
	\begin{claimproof}
		For a power $g_i^l$ to be in $G'$ we would certainly need $g_i^l\in \dot{G}$, so we'd need $M_i$ to divide $l$. Thus it remains to show that $\bar{h}_i=\bar{g}_i^{M_i}$ has order $Nr_i/M_i$.
		
		For each $g\in G$, we have that $g^{-1}h_i g$ is $\dot{G}$-conjugate to some $h_{ik}$. If $k\in K_i$ then $h_{ik}$ has order $\dot{N}_{ik}\dot{r}_i$ in $\dot{G}/\dot{G}'$, so $l=\dot{N}_{ik}\dot{r}_i$ is the least positive integer with $h_i^l\in g\dot{G}' g^{-1}$. If $k\notin K_i$ then $h_{ik}\in\dot{G}'$ and $h_i\in g\dot{G}' g^{-1}$.
		We conclude that $\bar{h}_i$ has order
		$$\LCM\{\dot{N}_{ik}\dot{r}_i\mid k\in K_i\}=\LCM\{\dot{N}_{ik}\mid k\in K_i\}\dot{r}_i=Nr_i/M_i.$$
	\end{claimproof}\\

For the second part of the lemma, suppose that the subgroups $\langle h_i\rangle$ have pairwise trivial intersections in the quotient $\dot{G}\to\dot{G}/\dot{G}'$. It remains to prove the following claim.\\

\begin{claim}
	The subgroups $\langle \bar{g}_i\rangle$ have pairwise trivial intersections.
\end{claim}\\

\begin{claimproof}
	Suppose for contradiction there is $i\neq j$ with $\langle \bar{g}_i\rangle\cap\langle \bar{g}_j\rangle$ non-trivial, say it contains an element $\bar{g}$ of order $p$, with $p$ prime. Then $p$ divides $|\langle \bar{g}_i\rangle|=Nr_i$ and $|\langle \bar{g}_j\rangle|=Nr_j$. Since $M_iM_j$ divides $N$, we deduce that $p$ divides at least one of $Nr_i/M_i$ and $Nr_j/M_j$, say it divides $Nr_i/M_i$. Then 
	$$\langle \bar{h}_i^{Nr_i/pM_i}\rangle=\langle \bar{g}_i^{Nr_i/p}\rangle=\langle\bar{g}\rangle,$$
	which implies that
	$$\{1\}\neq\langle\bar{g}\rangle< \langle \bar{g}_i\rangle\cap\langle \bar{g}_j\rangle\cap\dot{G}/G'.$$
	But $\langle \bar{g}_i\rangle\cap \dot{G}/G'=\langle \bar{h}_i\rangle$ and $\langle \bar{g}_j\rangle\cap \dot{G}/G'=\langle \bar{h}_j\rangle$, hence $\langle \bar{h}_i\rangle\cap\langle \bar{h}_j\rangle\neq\{1\}$. So $\langle h_i\rangle$ and $\langle h_j\rangle$ also have non-trivial intersection in $\dot{G}/\dot{G}'$, contrary to our assumption.
\end{claimproof}\\
\end{proof}

\bigskip
\section{Achieving command}\label{sec:buildingcommanding}

In this section we prove Theorem \ref{thm:command1}, which states that a virtually special cubulated group strongly commands any independent set of convex elements. We also prove Theorems \ref{thm:nonregc} and \ref{thm:weakcommandqc}, which are related to, but weaker than, the notion of commanding subgroups. Finally, we investigate when Theorem \ref{thm:command1} can be applied to right-angled Artin groups.

\subsection{Non-normal commanding}

The following theorem is a non-normal version of commanding subgroups since it does not require the finite-index subgroup $G'$ to be normal in $G$. Interpreted in terms of covers, this provides a way of controlling based elevations.

\begin{thm}(Non-normal commanding)\\\label{thm:nonregc}
	Let $X$ be a finite virtually special cube complex, and let $P_1,...,P_n<G:=\pi_1 X$ be convex subgroups with pairwise trivial intersections. Then there exist finite-index subgroups $\dot{P}_i<P_i$ such that for any further finite-index subgroups $P'_i<\dot{P}_i$ there exists a finite-index $G'<G$ with $P_i\cap G'=P'_i$.
\end{thm}
\begin{proof}
	Pick a basepoint $x\in X$ and write $G=\pi_1(X,x)$. By Lemma \ref{lem:convexisbased} there exist local isometries of finite cube complexes $\phi_i:(Y_i,y_i)\to(X,x)$ corresponding to the subgroups $P_i$. By Corollary \ref{cor:embedPi}, there is a finite directly special cover $(\hat{X},\hat{x})\to(X,x)$ such that the based elevations $\hat{\phi}_i:(\hat{Y}_i,\hat{y}_i)\to(\hat{X},\hat{x})$ of the $\phi_i$ are embedded and do not inter-osculate with hyperplanes of $\hat{X}$.
	We will consider $\hat{Y}_i$ as a subcomplex of $\hat{X}$ with $\hat{y}_i=\hat{x}$.
	 If $\hat{G}<G$ is the subgroup corresponding to $(\hat{X},\hat{x})\to(X,x)$ then $\hat{P}_i:=P_i\cap\hat{G}$ is the subgroup corresponding to $(\hat{Y}_i,\hat{x})\to(X,x)$.
	 
	 Let $\hat{G}_i$ be the imitator subgroup corresponding to $(\hat{Y}_i,\hat{x})\to(\hat{X},\hat{x})$, and let $\rho_i:\hat{G}_i\to\pi_1(\hat{Y}_i,\hat{x})=\hat{P}_i$ be the corresponding imitator homomorphism.
	 If the walker traverses a loop $\gamma$ in $\hat{Y}_i$ based at $\hat{x}$, then Subcomplex Entrapment implies that the $\hat{Y}_j$-imitator stays in $\hat{Y}_i\cap\hat{Y}_j$ for $i\neq j$, and so
	\begin{equation}\label{rhojhatKi}
		\rho_j(\hat{P}_i)\subset\hat{P}_i\cap\hat{P}_j=\{1\}.
	\end{equation}
    Now set
    $$\dot{P}_i:=\hat{P}_i\cap\bigcap_{j=1}^n \hat{G}_j.$$
	Given finite-index subgroups $P'_i<\dot{P}_i$, we can define a finite-index subgroup $G'<\dot{G}$ by
	$$G':=\bigcap_{i=1}^n \rho_i^{-1}(P'_i).$$
	We know that $P_i\cap\rho_i^{-1}(P'_i)=P'_i$ because $\rho_i$ is a retraction onto $\hat{P}_i$ (Construction \ref{cons:hom}) and $P'_i<\hat{P}_i$.
	And by (\ref{rhojhatKi}) we have $P_i\cap\rho_j^{-1}(P'_j)=\hat{P}_i$ for $i\neq j$.
	It follows that	$P_i\cap G'=P'_i$, as required.
\end{proof}

\subsection{Controlling imitator homomorphisms}\label{subsec:control}

We now turn to proving Theorem \ref{thm:rhoi}, which is a crucial step towards proving Theorem \ref{thm:command1}.
As Theorem \ref{thm:command1} involves a normal subgroup of $G$, we will need greater control over the imitator homomorphisms than we had in Theorem \ref{thm:nonregc}, and so we need to make use of the trivial wall projections from Theorem \ref{thm:trivialwallproj1}.

\begin{lem}\label{lem:prerhoi}
	Let $X$ be a finite virtually special cube complex, and let $P_1,...,P_n<G:=\pi_1 X$ be convex subgroups.
	Then there is a finite-index subgroup $\dot{G}< G$, and for each $i$ there is a homomorphism $\rho_i:\dot{G}\to P_i$ such that:
	\begin{enumerate}
		\item\label{item:hasid} $\rho_i$ is the identity on $\dot{P}_i:=P_i\cap\dot{G}$.
		\item\label{item:lKj=1} For each pair of indices $(i,j)$, if $P_i$ has trivial intersection with every conjugate of $P_j$ then $\rho_i(gP_j g^{-1}\cap\dot{G})=\{1\}$ for all $g\in G$.
	\end{enumerate}
\end{lem}
\begin{proof}
	Pick a basepoint $x\in X$ and write $G=\pi_1(X,x)$. By Lemma \ref{lem:convexisbased} there exist local isometries of finite cube complexes $\phi_i:(Y_i,y_i)\to(X,x)$ corresponding to the subgroups $P_i$. By Corollary \ref{cor:embedPi}, there is a finite regular directly special cover $(\hat{X},\hat{x})\to(X,x)$ such that the based elevations $(\hat{Y}_i,\hat{y}_i)\to(\hat{X},\hat{x})$ of the $\phi_i$ are embedded and do not inter-osculate with hyperplanes of $\hat{X}$.
	We will consider each $\hat{Y}_i$ as a subcomplex of $\hat{X}$ with $\hat{y}_i=\hat{x}$.
	Furthermore, by Theorem \ref{thm:trivialwallproj1} we can assume that, for each pair $(i,j)$, if $P_i$ has trivial intersection with every conjugate of $P_j$ then each elevation of $Y_i$ (not just the based elevation) has trivial wall projection onto each elevation of $Y_j$.
	If $\hat{G}\triangleleft G$ is the subgroup corresponding to $(\hat{X},\hat{x})\to(X,x)$ then $\hat{P}_i:=P_i\cap\hat{G}$ is the subgroup corresponding to $(\hat{Y}_i,\hat{x})\to(X,x)$. 
	
	Let $\hat{G}_i$ be the imitator subgroup corresponding to $(\hat{Y}_i,\hat{x})\to(\hat{X},\hat{x})$, and let $\rho_i:\hat{G}_i\to\pi_1(\hat{Y}_i,\hat{x})=\hat{P}_i$ be the corresponding imitator homomorphism. Set
	$$\dot{G}:=\bigcap_{i=1}^n \hat{G}_i.$$
	In particular, $\rho_i$ restricts to a homomorphism $\rho_i:\dot{G}\to\hat{P}_i<P_i$.
	It remains to verify properties \ref{item:hasid} and \ref{item:lKj=1} from the lemma.
	
	\begin{enumerate}
		\item We know from Construction \ref{cons:hom} that $\rho_i$ is a retraction onto $\hat{P}_i$, hence it restricts to the identity on $\dot{P}_i<\hat{P}_i$.
		
		\item Take a pair of indices $(i,j)$ such that $P_i$ has trivial intersection with every conjugate of $P_j$.		
		The cover $\hat{X}\to X$ is regular, so $G$ acts on $\hat{X}$ by deck transformations. Hence the elevations of $Y_j\to X$ to $\hat{X}$ are precisely the subcomplexes $g\hat{Y}_j$ with $g\in G$.	
		An element of $gP_j g^{-1}\cap\dot{G}<\hat{G}=\pi_1(\hat{X},\hat{x})$ corresponds to a loop that decomposes as a concatenation $\hat{\gamma}=\hat{\gamma}_g * \hat{\gamma}_j * \hat{\gamma}_g^{-1}$, where $\hat{\gamma}_g$ is a path in $\hat{X}$ from $\hat{x}$ to $g\hat{x}$ that projects to a loop in $X$ representing the element $g$, and $\hat{\gamma}_j$ is a loop in $g\hat{Y}_j$.
		If the walker traverses the loop $\hat{\gamma}$, then the imitator in $\hat{Y}_i$ traverses a loop $\hat{\delta}$ in $\hat{Y}_i$ based at $\hat{x}$, which represents the element $\rho_i([\hat{\gamma}])$. Moreover, the decomposition of $\hat{\gamma}$ will induce a decomposition $\hat{\delta}=\hat{\delta}_g * \hat{\delta}_j * \hat{\delta}_g^{-1}$.
		By construction of $\hat{X}$, we know that $\WProj_{\hat{X}}(g\hat{Y}_j\to\hat{Y}_i)$ is trivial; so, as the walker traverses the loop $\hat{\gamma}_j$ in $g\hat{Y}_j$, the imitator must traverse a null-homotopic loop in $\hat{Y}_i$, which will be the loop $\hat{\delta}_j$. It follows that the whole loop $\hat{\delta}$ is null-homotopic and that $\rho_i(gP_j g^{-1}\cap\dot{G})=\{1\}$.\qedhere	
	\end{enumerate}
\end{proof}

We now upgrade Lemma \ref{lem:prerhoi} to cubulated groups (which might have torsion).

\theoremstyle{plain}
\newtheorem*{thm:rhoi}{Theorem \ref{thm:rhoi}}
\begin{thm:rhoi}(Independent imitator homomorphisms)\\
	Let $G\acts X$ be a virtually special cubulated group, and let $P_1,...,P_n<G$ be convex subgroups, such that all conjugates of $P_i$ and $P_j$ have finite intersection if $i\neq j$. Then there is a finite-index normal subgroup $\dot{G}\triangleleft G$, and for each $i$ there is a homomorphism $\rho_i:\dot{G}\to P_i$ such that:
	\begin{enumerate}
		\item\label{item:hasid2} $\rho_i$ is the identity on $\dot{P}_i:=P_i\cap\dot{G}$.
		\item\label{item:lKj=12} $\rho_i(gP_j g^{-1}\cap\dot{G})=\{1\}$ for $g\in G$ and $i\neq j$.
	\end{enumerate}
\end{thm:rhoi}
\begin{proof}
	Let $\hat{G}\triangleleft G$ be a finite-index torsion-free subgroup. Then $X/\hat{G}$ is a finite virtually special cube complex. Write $\hat{P}_i:=P_i\cap\hat{G}$. 
	Note that the $G$-conjugacy class of $\hat{P}_i$ may split into several $\hat{G}$-conjugacy classes.
	Apply Lemma \ref{lem:prerhoi} to $\hat{G}$ and a collection of subgroups $\hat{\calP}$ consisting of the $\hat{P}_i$ together with a (finite) collection of $G$-conjugates of the $\hat{P}_i$ that represent all possible $\hat{G}$-conjugacy classes.
	We obtain a finite-index subgroup $\dot{G}<\hat{G}$ and homomorphisms from $\dot{G}$ to each subgroup in $\hat{\calP}$. We are only interested in the homomorphisms to the subgroups $\hat{P}_i$ however, which we denote $\rho_i:\dot{G}\to \hat{P}_i$.
	
	Property \ref{item:hasid2}, that $\rho_i$ is the identity on $\dot{P}_i:=P_i\cap\dot{G}$, follows immediately from Lemma \ref{lem:prerhoi}\ref{item:hasid}.
	We now verify property \ref{item:lKj=12}, so let $g\in G$ and $i\neq j$.
	Write $gP_j g^{-1}\cap\hat{G}= \hat{g}\hat{P}\hat{g}^{-1}$ for some $\hat{g}\in\hat{G}$ and $\hat{P}\in\hat{\calP}$.
	Then $gP_j g^{-1}\cap\dot{G}=\hat{g}\hat{P}\hat{g}^{-1}\cap\dot{G}$.
	All $G$-conjugates of $P_i$ and $P_j$ have finite intersection, so all $\hat{G}$-conjugates of $\hat{P}_i$ and $\hat{P}$ have trivial intersection (remember that $\hat{G}$ is torsion-free).
	It then follows from Lemma \ref{lem:prerhoi}\ref{item:lKj=1} that $\rho_i(gP_j g^{-1}\cap\dot{G})=\{1\}$, as required. Finally, we may assume that $\dot{G}$ is normal in $G$ by passing to the normal core (properties \ref{item:hasid2} and \ref{item:lKj=12} will clearly still hold).
\end{proof}

\subsection{Towards commanding subgroups}

Theorem \ref{thm:rhoi} can be used to deduce the following theorem, which is a weak form of a virtually special cubulated group $G\acts X$ commanding an almost malnormal collection of convex subgroups $(P_i)$. Like the commanding property, it allows us to construct a finite-index normal subgroup $G'\triangleleft G$ and control the intersections $P_i\cap G'$ independently; but, unlike commanding, $P_i\cap G'$ must belong to a fixed sequence $(P_{il})$ of subgroups of $P_i$.
The author believes upgrading this theorem to the full strength of commanding is a difficult task, but there are intermediate upgrades that would already be very useful; for instance if each sequence $(P_{il})$ only depended on $P_i$ as an abstract group, then the theorem would provide a powerful tool for constructing finite covers of graphs of groups for which each vertex group satisfies the assumptions of the theorem with respect to its incident edge groups (see Proposition \ref{prop:GoG}).

\begin{thm}\label{thm:weakcommandqc}
	Let $G\acts X$ be a virtually special cubulated group, and let $P_1,...,P_n<G$ be convex subgroups such that all conjugates of $P_i$ and $P_j$ have finite intersection if $i\neq j$. Then for each $i$ there is a descending sequence of finite-index subgroups $P_i>P_{i1}>P_{i2}>\cdots$ with trivial intersection, such that for any map $\sigma:\{1,...,n\}\to\mathbb{N}$ there exists a finite-index normal subgroup $G'\triangleleft G$ with $P_i\cap G'=P_{i\sigma(i)}$.
\end{thm}
\begin{proof}
	Apply Theorem \ref{thm:rhoi}. For each $i$, choose a descending sequence of finite-index subgroups $\dot{P}_i>\dot{P}_{i1}>\dot{P}_{i2}>...$ with trivial intersection, that are all normal in $P_i$ (such a sequence exists by residual finiteness of $P_i$). Set
	$$G_{il}:=\bigcap_{g\in G}g\rho_i^{-1}(\dot{P}_{il})g^{-1}\triangleleft G.$$
	$G_{il}$ has finite index in $G$ as it is the normal core of a finite-index subgroup of $G$.
	Now define
	$$P_{il}:=\dot{P}_i\cap G_{il}.$$
	$P_{il}$ has finite index in $P_i$, and the sequence $(P_{il})_{l\in\N}$ is clearly descending.
	Theorem \ref{thm:rhoi}\ref{item:hasid2} implies that $P_{il}<\dot{P}_{il}$, so $\cap_l P_{il}=\{1\}$.	
	Given $\sigma:\{1,...,n\}\to\mathbb{N}$, set
	$$G':=\bigcap_{i=1}^n G_{i\sigma(i)}\triangleleft G.$$
	This is a finite-index normal subgroup of $G$ because the $G_{i\sigma(i)}$ are.
	Finally, by Theorem \ref{thm:rhoi}\ref{item:lKj=12} we have
	\begin{align*}
		P_i\cap G'&=\dot{P}_i\cap G'\\
		&=\dot{P}_i\cap\bigcap_{j=1}^n\bigcap_{g\in G}g\rho_j^{-1}(\dot{P}_{j\sigma(j)})g^{-1}\\
		&=\dot{P}_i\cap \bigcap_{g\in G}g\rho_i^{-1}(\dot{P}_{il})g^{-1}\\
		&=P_{i\sigma(i)}.\qedhere
	\end{align*}
\end{proof}

\subsection{Proof of Theorem \ref{thm:command1}}\label{subsec:thmcommand1}

Recall that a collection of group elements $\{g_1,...,g_n\}$ is \emph{independent} if the elements $g_i$ have infinite order and no non-zero power of $g_i$ is conjugate to a non-zero power of $g_j$ for $i\neq j$.

\theoremstyle{plain}
\newtheorem*{thm:command1}{Theorem \ref{thm:command1}}
\begin{thm:command1}
	Every virtually special cubulated group $G\acts X$ strongly commands every independent set of convex elements.
\end{thm:command1}
\begin{proof}
Let $\{g_1,...,g_n\}$ be an independent set of convex elements.
Apply Theorem \ref{thm:rhoi} to $G$ and the subgroups $P_i:=\langle g_i\rangle$ to obtain a finite-index normal subgroup $\dot{G}\triangleleft G$ and homomorphisms $\rho_i:\dot{G}\to P_i$. Suppose $g_i$ has order $M_i$ in $G/\dot{G}$ and put $h_i:=g_i^{M_i}$. The $G$-conjugacy class of $h_i$ may split into several $\dot{G}$-conjugacy classes, let $\{h_i=h_{i1},...,h_{in_i}\}$ be a set of representatives. We would like to apply Lemma \ref{lem:commandfromsubgp} to deduce that $G$ strongly commands $\{g_i\}$, so it remains to show that the elements $\{h_{ik}\}\subset\dot{G}$ satisfy the hypothesis from this lemma. Namely we must define subcollections $\{h_{ik}\mid k\in K_i\}$ and integers $\dot{N}_{ik}>0$ for $k\in K_i$ such that for any integers $\dot{r}_i>0$ there exists a finite quotient $\dot{G}\to\dot{G}/\dot{G}'$ where $h_{ik}$ has order $\dot{N}_{ik}\dot{r}_i$ for $k\in K_i$, and $h_{ik}\in\dot{G}'$ otherwise. And we need the images of the subgroups $\langle h_i\rangle$ in the quotient $\dot{G}/\dot{G}'$ to have pairwise trivial intersections.

We start by considering the product of the homomorphisms $\rho_i$.
$$\rho:=(\rho_1,...,\rho_n):\dot{G}\longrightarrow P_1\times\cdots\times P_n \cong \Z^n.$$

Properties \ref{item:hasid2} and \ref{item:lKj=12} from Theorem \ref{thm:rhoi} become
\begin{enumerate}
	\item\label{item:rhoihi} $\rho(h_i)=(0,...,0,M_i,0,...,0)$,
	\item\label{item:rhoihjl} $\rho(h_{ik})=(0,...,0,A_{ik},0,...,0)$,
\end{enumerate}
for some integers $A_{ik}$, where $M_i$ and $A_{ik}$ appear in the $i$-th coordinate above.
For each $i$ define the set
$$K_i:=\{k\mid A_{ik}\neq0,\,1\leq k\leq n_i\}.$$
We know that $A_{i1}=M_i$, so $1\in K_i$. Now define
$$B_i:=\prod_{k\in K_i} A_{ik},$$
and for $k\in K_i$ define
$$\dot{N}_{ik}=|B_i/A_{ik}|.$$
Given integers $\dot{r}_i>0$, we define the homomorphism $\rho'$ as the following composition
$$\rho': \dot{G} \overset{\rho}{\longrightarrow} \Z^n \longrightarrow \frac{\Z}{B_1 \dot{r}_1 \Z} \times \cdots\times \frac{\Z}{B_n \dot{r}_n \Z}.$$

It follows from the definitions of $B_i$ and $\dot{N}_{ik}$ that $\rho'(h_{ik})$ has order $\dot{N}_{ik}\dot{r}_i$ for $k\in K_i$, and $\rho'(h_{ik})$ is the identity for $k\notin K_i$.
It is also clear that $\rho'(\langle h_i\rangle)\cap \rho'(\langle h_j\rangle)$ is trivial for $i\neq j$.
Thus the required subgroup $\dot{G}' \triangleleft \dot{G}$ is the kernel of $\rho'$.
\end{proof}

\subsection{Right-angled Artin groups}

Right-angled Artin groups (or RAAGs) are a good source of virtually special cubulated groups that are well understood but not hyperbolic, so they are the perfect setting to apply Theorem \ref{thm:command1}. In this subsection we investigate when elements of RAAGs are independent or convex, and in Theorem \ref{thm:RAAG} we deduce that RAAGs command random collections of elements.

Given a finite simplicial graph $\Gamma$, the \emph{right-angled Artin group} (or \emph{RAAG}) $G_\Gamma$ has generators corresponding to the vertex set $V$ of $\Gamma$ and relations $[v_1,v_2]$ whenever $v_1,v_2\in V$ are adjacent in $\Gamma$. There are very explicit solutions to the word and conjugacy problems in $G_\Gamma$, which we now describe. These solutions are implicit in the work of Servatius \cite{Servatius89} and Hermiller--Meier \cite{HermillerMeier95}.

Write $V^\pm$ for the set of generators and their inverses. An element $g\in G_\Gamma$ can be represented by a word $w$ on the letters $V^\pm$, and we say that $w$ is \emph{reduced} if its length $|w|$ is minimal among all words representing $g$. A solution to the word problem is given by the following proposition.

\begin{prop}\label{prop:wordsol}
If $w,w'$ are words representing $g\in G_\Gamma$, with $w'$ reduced, then we can transform $w$ into $w'$ via a sequence of the following two moves:
\begin{enumerate}[label={(M\arabic*)}]
	\item\label{item:cancel} Remove a subword $v^{-1}v$ with $v\in V^\pm$.
	\item\label{item:commute} Replace a subword $v_1v_2$ with $v_2v_1$ where $v_1,v_2\in V^\pm$ correspond to adjacent vertices.
\end{enumerate}
\end{prop}

In particular, it follows that $v_1,v_2\in V^\pm$ commute in $G_\Gamma$ if and only if their corresponding vertices are equal or adjacent. It also follows that any two reduced words representing $g$ differ by \ref{item:commute} moves.

We now turn to the conjugacy problem. We say that a word $w$ representing $g\in G_\Gamma$ is \emph{cyclically reduced} if it has minimal length among words representing conjugates of $g$.

\begin{prop}\label{prop:conjsol}
	If $w,w'$ are words representing conjugate elements in $G_\Gamma$, with $w'$ cyclically reduced, then we can transform $w$ into $w'$ using the moves \ref{item:cancel}, \ref{item:commute} and the following cyclic permutation move:
\begin{enumerate}[label={(M3)}]
	\item\label{item:permute} Replace the word $vw$ with $wv$, where $v\in V^\pm$.
\end{enumerate}
Moreover, $w$ is cyclically reduced if and only if it cannot be transformed via \ref{item:commute} moves to a word of the form $vuv^{-1}$ with $v\in V^\pm$.
\end{prop}

In particular, any two cyclically reduced words representing conjugate elements differ by \ref{item:commute} and \ref{item:permute} moves. And it follows from the moreover part of Proposition \ref{prop:conjsol} that any power of a cyclically reduced word is cyclically reduced.

We are now ready to describe an efficient algorithm to check independence of elements.

\begin{prop}\label{prop:indepsol}
	There is an explicit algorithm that determines whether a collection of elements in $G_\Gamma$ is independent.
\end{prop}
\begin{proof}
	It suffices to determine whether a pair of non-trivial elements $\{g_1,g_2\}$ is independent. Let $w_1,w_2$ be cyclically reduced words representing conjugates of $g_1,g_2$. Suppose that $\{g_1,g_2\}$ is not independent. Then there exist non-zero powers $w_1^m,w_2^n$ that represent conjugate elements. We know that $w_1^m,w_2^n$ are also cyclically reduced, hence they differ by \ref{item:commute} and \ref{item:permute} moves -- in particular $m|w_1|=|w_1^m|=|w_2^n|=n|w_2|$. Furthermore, if $m,n$ have a common factor, say $m=Md,n=Nd$, then we deduce that $w_1^M,w_2^N$ also represent conjugate elements (we always have $g^d=h^d$ implies $g=h$ for elements of $G_\Gamma$ by Proposition \ref{prop:wordsol}).
	Thus we obtain the following algorithm:
	\begin{enumerate}
		\item Find cyclically reduced words $w_1,w_2$ representing conjugates of $g_1,g_2$ using using \ref{item:cancel}--\ref{item:permute} moves.
		\item Take non-zero integers $m,n$ with $m|w_1|=n|w_2|$.
		\item Check whether $w_1^m$ can be transformed into either $w_2^n$ or $w_2^{-n}$ via \ref{item:commute}--\ref{item:permute} moves. $\{g_1,g_2\}$ is independent if and only if the answer is no.\qedhere
	\end{enumerate} 
\end{proof}

Next we turn to convexity of elements in RAAGs. $G_\Gamma$ is the fundamental group of a finite special cube complex $S_\Gamma$, called the Salvetti complex -- see \cite{HaglundWise08}. Let $X_\Gamma$ be the universal cover of $S_\Gamma$, so $X_\Gamma$ is a CAT(0) cube complex with a free cocompact $G_\Gamma$ action. 
For $g\in G_\Gamma$, let $\Gamma(g)$ consist of those $v\in V$ such that $v$ or $v^{-1}$ appears in some (equivalently any) cyclically reduced word representing an element conjugate to $g$. The following Proposition was proved by Fioravanti \cite[Lemma 3.10(2)]{Fioravanti21}.

\begin{prop}\label{prop:RAAGconvex}
	Let $1\neq g\in G_\Gamma$. Then $g$ is convex with respect to $X_\Gamma$ if and only if the full subgraph spanned by $\Gamma(g)$ does not split as a non-trivial join.
\end{prop}

Armed with the previous two propositions, we can now show that RAAGs command collections of elements produced by randomly picking long words on $V^\pm$.

\begin{thm}(RAAGs command random elements)\\\label{thm:RAAG}
Fix $n$, and let $w_1,...,w_n$ be words on $V^\pm$ chosen uniformly at random from among all freely reduced words of length at most $L$, and suppose they represent elements $g_1,...,g_n\in G_\Gamma$. Then $G_\Gamma$ commands $\{g_1,...,g_n\}$ with probability converging to 1 as $L\to\infty$ -- unless $\Gamma$ is a complete graph with fewer than $n$ vertices.
\end{thm}
\begin{proof}
	If $\Gamma$ is a complete graph with at least $n$ vertices then $G_\Gamma$ is a free abelian group of rank at least $n$. In this case the theorem follows from Proposition \ref{prop:abeliancommand}. 
	
	Now suppose that $\Gamma$ is not complete. For $v\in V$ let $\pi_v:G_\Gamma\to\mathbb{Z}$ be the homomorphism that maps $v\mapsto1$ and $v'\mapsto0$ for $v\neq v'\in V$.
	Note that $\pi_v(g_i)$ is the number of occurrences of $v$ in the word $w_i$ subtract the number of occurrences of $v^{-1}$.	
	 We will show that $G_\Gamma$ commands $\{g_1,...,g_n\}$ provided the following two conditions are satisfied. These conditions are clearly satisfied for long random words $w_1,...,w_n$, i.e. with probability converging to 1 as $L\to\infty$. 
	\begin{enumerate}
		\item\label{item:pinonzero} $\pi_v(g_i)\neq0$ for all $v\in V$ and $1\leq i\leq n$.
		\item\label{item:distinctratios} For each pair $\{v,v'\}\subset V$, the ratios $(\pi_v(g_i)/\pi_{v'}(g_i)\mid 1\leq i\leq n)$ are distinct.
	\end{enumerate}

Take a maximal partition $V=\sqcup_{k=1}^q V_k$ such that any two vertices in distinct $V_k$ are joined by an edge. If the induced subgraph of $V_k$ is $\Gamma_k$, then we get a product decomposition
\begin{equation}\label{GGaprod}
G_\Gamma=\prod_{k=1}^q G_{\Gamma_k}.
\end{equation}
Since $\Gamma$ is not complete, $|V_k|>1$ for some $k$, so suppose $|V_1|>1$. For $v\in V_1$ we can define a homomorphism $\pi_v:G_{\Gamma_1}\to\mathbb{Z}$ as we did for $G_\Gamma$. Let $\mu:G_\Gamma\to G_{\Gamma_1}$ be the projection map onto the factor $G_{\Gamma_1}$ with respect to the product decomposition (\ref{GGaprod}). Note that $\pi_v=\pi_v\mu:G_\Gamma\to\mathbb{Z}$, hence conditions \ref{item:pinonzero} and \ref{item:distinctratios} apply to $\mu(g_1),...,\mu(g_n)$ and vertices in $V_1$.

Condition \ref{item:pinonzero} implies that $\Gamma_1(\mu(g_i))=V_1$ for all $i$, so Proposition \ref{prop:RAAGconvex} and the maximality of the partition $\sqcup_k V_k$ implies that the elements $\mu(g_i)$ are convex with respect to $X_{\Gamma_1}$. The ratios $\pi_v(g_i)/\pi_{v'}(g_i)$ are preserved by taking conjugates and powers, so the collection $\{\mu(g_1),...,\mu(g_n)\}$ is independent by condition \ref{item:distinctratios}. Thus $G_{\Gamma_1}$ commands $\{\mu(g_1),...,\mu(g_n)\}$ by Theorem \ref{thm:command1}. By composing $\mu$ with suitable homomorphisms from $G_{\Gamma_1}$ to finite groups, it follows that $G_\Gamma$ commands $\{g_1,...,g_n\}$.
\end{proof}

\bigskip
\section{Virtually connected intersections}\label{sec:hierarchical}

The aim of this section is to prove Theorem \ref{thm:connected1}. In order to do this we will generalise the walker-imitator construction from Section \ref{sec:hom} to a hierarchy of imitators.

\subsection{A hierarchy of imitators}\label{subsec:hierarchies}

\begin{cons}(A hierarchy of imitators)\\\label{cons:hierarchy}
	Let $Y_1,...,Y_n\to X$ be local isometries of finite directly special cube complexes.
	We consider the walker wandering around the 1-skeleton of $X$ while a set of imitators wander around the 1-skeleta of the $Y_i$.
	We index the imitators by finite sequences $(i_0,...,i_m)$ of integers in $\{1,...,n\}$, and let $\Sigma$ be the set of all such sequences. Imitator $(i_0,...,i_m)$ wanders around the 1-skeleton of $Y_{i_0}$. 	 
	Each imitator either tries to copy the walker or tries to copy another imitator.
	Imitator $(i_0)$ tries to copy the walker, with exactly the same movement rule as in Construction \ref{cons:imitator}.
	Imitator $(i_0,...,i_m)$ tries to copy imitator $(i_1,...,i_m)$ ($m>0$) in essentially the same way:
	we project the movements of imitator $(i_1,...,i_m)$ to $X$ and imitator $(i_0,...,i_m)$ tries to copy these movements precisely as in Construction \ref{cons:imitator}. 
	More explicitly, if imitators $(i_0,...,i_m)$ and $(i_1,...,i_m)$ are at vertices $(y_0,y_1)\in Y^0_{i_0}\times Y^0_{i_1}$ and imitator $(i_1,...,i_m)$ traverses the edge $f_1\in\link(y_1)$, then imitator $(i_0,...,i_m)$ traverses the edge $f_0\in\link(y_0)$ with $\phi_{i_0}(f_0)\parallel\phi_{i_1}(f_1)$, if such an edge exists, otherwise they remain at $y_0$.	
	Each movement of the walker causes all of the imitators to move (or not move) by working up the hierarchy.
	
	Let
	$$\Theta=\{\theta:\Sigma\to\sqcup_i Y^0_i\mid\theta(i_0,...,i_m)\in Y_{i_0}\}$$
	denote the set of possible positions of the imitators.
	If the imitators start at position $\theta\in\Theta$ and the walker
	traverses a path $\gamma$, then imitator $(i_0,...,i_m)$ traverses a path that we denote by $\delta_{(i_0,...,i_m)}(\gamma,\theta)$. As in Construction \ref{cons:imitator}, the process is reversible, so if the walker immediately backtracks along $\gamma$ then imitator $(i_0,...,i_m)$ will backtrack along the path $\delta_{(i_0,...,i_m)}(\gamma,\theta)$. We denote the terminal vertex of $\delta_{(i_0,...,i_m)}(\gamma,\theta)$ by $t_{(i_0,...,i_m)}(\gamma,\theta)$.
\end{cons}

\begin{lem}\label{lem:hgammasquare}
	If $\gamma$ and $\gamma'$ are homotopic paths in $X$, then for any $\theta\in\Theta$ and $(i_0,...,i_m)\in\Sigma$ the paths $\delta_{(i_0,...,i_m)}(\gamma,\theta)$ and $\delta_{(i_0,...,i_m)}(\gamma',\theta)$ are homotopic in $Y_{i_0}$.
\end{lem}
\begin{proof}
We induct on $m$. If $m=0$ then imitator $(i_0)$ tries to copy the walker exactly as in Construction \ref{cons:imitator}, so the result is immediate from Lemma \ref{lem:gammasquare}. For $m>0$, we assume by induction that $\delta_{(i_1,...,i_m)}(\gamma,\theta)$ is homotopic to 
$\delta_{(i_1,...,i_m)}(\gamma',\theta)$. These paths are the movements of imitator $(i_1,...,i_m)$ if the walker traverses $\gamma$ or $\gamma'$. But imitator $(i_0,...,i_m)$ tries to copy the projection to $X$ of the movements of imitator $(i_1,...,i_m)$ in the same way as Construction \ref{cons:imitator}, so applying Lemma \ref{lem:gammasquare} again tells us that the homotopy class of the movements of imitator $(i_1,...,i_m)$ determines the homotopy class of the movements of imitator $(i_0,...,i_m)$. The result follows.
\end{proof}

\begin{cons}(Hierarchical imitator subgroups and homomorphisms)\\\label{cons:hsubgroup}
	Retaining the setup of Construction \ref{cons:hierarchy}, pick a basepoint $x\in X^0$. We have a right-action of $G:=\pi_1(X,x)$ on $\Theta$ defined
	by $(\theta\cdot[\gamma])(i_0,...,i_m):=t_{(i_0,...,i_m)}(\gamma,\theta)$ for $\theta\in\Theta$, $[\gamma]\in G$ and $(i_0,...,i_m)\in\Sigma$ -- this only depends on the homotopy class of $\gamma$ by Lemma \ref{lem:hgammasquare}. Let $G_\theta<G$ be the stabiliser of $\theta\in\Theta$. In general $G_\theta$ might not have finite index in $G$, even if the $Y_i$ are finite (unlike in Construction \ref{cons:hom}). 
	
	We can then define a homomorphism $$\rho_{x,\theta,(i_0,...,i_m)}:G_\theta\to\pi_1(Y_{i_0},\theta(i_0,...,i_m))$$
	by $\rho_{x,\theta,(i_0,...,i_m)}([\gamma]):=[\delta_{(i_0,...,i_m)}(\gamma,\theta)]$ -- again this is well-defined by Lemma \ref{lem:hgammasquare}. Unlike in Construction \ref{cons:hom}, $\rho_{x,\theta,(i_0,...,i_m)}$ might not be a retraction onto $\pi_1(Y_{i_0},\theta(i_0,...,i_m))$, indeed $\pi_1(Y_{i_0},\theta(i_0,...,i_m))$ might not even be a subgroup of $G_\theta$. (Note: The proof of Theorem \ref{thm:connected1} only uses the subgroup $G_\theta$, the definition of $\rho_{x,\theta,(i_0,...,i_m)}$ is just to maintain the parallel with Construction \ref{cons:hom}).
\end{cons}

\subsection{Proof of Theorem \ref{thm:connected1}}

We recall Theorem \ref{thm:connected1}. This generalises a theorem of Haglund and Wise \cite[Theorem 4.25]{HaglundWise12} to the non-hyperbolic setting.

\theoremstyle{plain}
\newtheorem*{thm:connected1}{Theorem \ref{thm:connected1}}
\begin{thm:connected1}(Virtually connected intersections)\\
		For $i=1,...,n$ let $(Y_i,y_i)\to(X,x)$ be based local isometries of finite virtually special cube complexes. Then there is a finite cover $(\dot{X},\dot{x})\to(X,x)$ such that the based elevations $\dot{Y}_i$ of $Y_i$ are embedded in $\dot{X}$ and have connected intersections $\cap_{i\in E}\dot{Y}_i$ for any $\emptyset\neq E\subset\{1,...,n\}$. Moreover, if the $Y_i$ are embedded in $X$ and do not inter-osculate with hyperplanes of $X$, then we may assume that the covering map $\dot{X}\to X$ is injective on $\cap_{i=1}^n \dot{Y}_i$.  
\end{thm:connected1}

Using Corollary \ref{cor:embedPi}, we reduce to the case where the $Y_i$ are embedded in $X$ and do not inter-osculate with hyperplanes. We will consider the $Y_i$ as subcomplexes of $X$ with $y_i=x$.
We now apply the hierarchy of imitators from Section \ref{subsec:hierarchies}. Our arguments will be restricted to the case where all imitators start at the vertex $x$, so let $\theta_x\in \Theta$ denote the constant function at $x$. We will use the abbreviations
\begin{align*}
\delta_{(i_0,...,i_m)}(\gamma)&:=\delta_{(i_0,...,i_m)}(\gamma,\theta_x),\\
t_{(i_0,...,i_m)}(\gamma)&:=t_{(i_0,...,i_m)}(\gamma,\theta_x).
\end{align*}
Let $(\dot{X},\dot{x})\to(X,x)$ be the cover corresponding to the subgroup $G_{\theta_x}<G:=\pi_1(X,x)$. We will spend the rest of this section proving that this cover satisfies the properties of Theorem \ref{thm:connected1} (including finiteness).

\begin{lem}\label{lem:splitdelta}
	$\delta_{(i_0,...,i_m)}(\gamma)=\delta_{(i_0,...,i_p)}(\delta_{(i_{p+1},...,i_m)}(\gamma))$ for $0\leq p<m$.
\end{lem}
\begin{proof}
	This equality follows straight from Construction \ref{cons:hierarchy} given the observation that the path taken by any given imitator does not depend on the name of the imitator but only depends on their starting point and on the path taken by the walker/imitator whom they try to copy.
\end{proof}

\begin{lem}\label{lem:addj}
	If $\gamma$ is a path in $Y_j$ based at $x$ then $\delta_{(i_0,...,i_m,j)}(\gamma)=\delta_{(i_0,...,i_m)}(\gamma)$ for all $(i_0,...,i_m)\in\Sigma$.
\end{lem}
\begin{proof}
	As $\gamma$ is in $Y_j$, if the walker traverses $\gamma$ then imitator $(j)$ also traverses $\gamma$, so $\delta_{(j)}(\gamma)=\gamma$. Applying Lemma \ref{lem:splitdelta} for $(i_0,...,i_m)\in\Sigma$ then gives us
	\begin{align*}
\delta_{(i_0,...,i_m,j)}(\gamma)&=\delta_{(i_0,...,i_m)}(\delta_{(j)}(\gamma))\\
&=\delta_{(i_0,...,i_m)}(\gamma).\qedhere
	\end{align*}
\end{proof}

\begin{lem}\label{lem:inintersection}
	$\delta_{(i_0,...,i_m)}(\gamma)$ lies in $\cap_{s=0}^m Y_{i_s}$.
\end{lem}
\begin{proof}
	We prove this by induction on $m$. It is immediate for $m=0$, so suppose $m>0$.
	Imitator $(i_0,...,i_m)$ tries to copy imitator $(i_1,...,i_m)$, and they traverse the path $\delta_{(i_1,...,i_m)}(\gamma)$, which lies in $\cap_{s=1}^m Y_{i_s}$ by induction. Both imitators start at $x\in\cap_{i=1}^n Y_i$, and the subcomplexes $Y_i$ do not inter-osculate with hyperplanes, so it follows from Subcomplex Entrapment (Lemma \ref{lem:subcomplextrap}) that imitator $(i_0,...,i_m)$ stays inside $\cap_{s=0}^m Y_{i_s}$.
\end{proof}

For $(i_0,...,i_m)\in\Sigma$ let $\nu(i_0,...,i_m)$ be obtained by keeping the rightmost occurrence of each number $i\in\{1,...,n\}$ and deleting all other $i_s$ (e.g. $\nu(1,2,2,4,1,4)=(2,1,4)$). This defines a map
$$\nu:\Sigma\to\Sigma_d:=\{(i_0,...,i_m)\in\Sigma\mid i_0,...,i_m\text{ distinct}\}.$$

\begin{lem}\label{lem:nu}
$\delta_{\nu(i_0,...,i_m)}(\gamma)=\delta_{(i_0,...,i_m)}(\gamma)$, so in particular $t_{\nu(i_0,...,i_m)}(\gamma)=t_{(i_0,...,i_m)}(\gamma)$.
\end{lem}
\begin{proof}
	Suppose $0\leq p<q\leq m$ with $i_p=i_q$. It suffices to show that
	\begin{equation*}\label{removep}
\delta_{(i_0,...,i_{p-1},i_{p+1},...,i_m)}(\gamma)=\delta_{(i_0,...,i_m)}(\gamma),
	\end{equation*}
	since $\nu(i_0,...,i_m)$ is obtained by a sequence of such deletions of indices.
	Lemma \ref{lem:inintersection} tells us that $\delta_{(i_{p+1},...,i_m)}(\gamma)$ lies inside $Y_{i_p}=Y_{i_q}$, so we get the following chain of equalities
	\begin{align*}
\delta_{(i_0,...,i_{p-1},i_{p+1},...,i_m)}(\gamma)&=\delta_{(i_0,...,i_{p-1})}(\delta_{(i_{p+1},...,i_m)}(\gamma))&\text{by Lemma \ref{lem:splitdelta},}\\
&=\delta_{(i_0,...,i_p)}(\delta_{(i_{p+1},...,i_m)}(\gamma))&\text{by Lemma \ref{lem:addj},}\\
&=\delta_{(i_0,...,i_m)}(\gamma)&\text{by Lemma \ref{lem:splitdelta}.} \tag*{\qedhere}
	\end{align*}
\end{proof}

\begin{lem}
$(\dot{X},\dot{x})\to(X,x)$ is a finite cover.
\end{lem}
\begin{proof}
	This is equivalent to $G_{\theta_x}$ having finite index in $G$, which in turn is equivalent to the orbit $\theta_x\cdot G$ being finite. For $[\gamma]\in G$, Lemma \ref{lem:nu} implies that
	$$\theta_x\cdot[\gamma]:(i_0,...,i_m)\mapsto t_{(i_0,...,i_m)}(\gamma)=t_{\nu(i_0,...,i_m)}(\gamma).$$
	So $\theta_x\cdot[\gamma]$ is determined by its restriction to $\Sigma_d$. But $\Sigma_d$ and the $Y_i$ are finite, thus there are only finitely many possibilities for $\theta_x\cdot[\gamma]$.
\end{proof}

\begin{lem}\label{lem:sameendpt}
	Let $\gamma_1,\gamma_2$ be paths in $X$ based at $x$ with the same terminal vertex that lift to paths $\dot{\gamma}_1,\dot{\gamma}_2$ in $\dot{X}$ based at $\dot{x}$. Then $\dot{\gamma}_1,\dot{\gamma}_2$ have the same terminal vertex if and only if $t_{(i_0,...,i_m)}(\gamma_1)=t_{(i_0,...,i_m)}(\gamma_2)$ for all $(i_0,...,i_m)\in\Sigma$.
\end{lem}
\begin{proof}
	Observe that $\dot{\gamma}_1,\dot{\gamma}_2$ have the same endpoint if and only if $\gamma_1 * \gamma_2^{-1}$ lifts to a loop in $\dot{X}$ based at $\dot{x}$, which is in turn equivalent to $[\gamma_1 * \gamma_2^{-1}]\in G_{\theta_x}$. That is equivalent to $\theta_x\cdot[\gamma_1]=\theta_x\cdot[\gamma_2]$, which is equivalent to $t_{(i_0,...,i_m)}(\gamma_1)=t_{(i_0,...,i_m)}(\gamma_2)$ for all $(i_0,...,i_m)\in\Sigma$.
\end{proof}

We complete the proof of Theorem \ref{thm:connected1} with the following two lemmas.

\begin{lem}\label{lem:connected}
	Let $\dot{Y}_i$ be the based elevation of $Y_i$ to $(\dot{X},\dot{x})$ and let $\emptyset\neq E\subset\{1,...,n\}$. Then $\cap_{i\in E}\dot{Y}_i$ is connected. 
\end{lem}
\begin{proof}
	Without loss of generality, suppose $E=\{1,...,k\}$ for some $1\leq k\leq n$.
	We know that $\dot{x}\in\cap_{i=1}^k\dot{Y}_i$, so consider another vertex $\dot{x}'\in\cap_{i=1}^k\dot{Y}_i$.
	Our task is to find a path in $\cap_{i=1}^k\dot{Y}_i$ joining $\dot{x}$ and $\dot{x}'$.
	Let $\dot{\gamma}_1,...,\dot{\gamma}_k$ be paths in $\dot{Y}_1,...,\dot{Y}_k$ respectively that go from $\dot{x}$ to $\dot{x}'$. Say they project to paths $\gamma_1,...,\gamma_k$ in $X$.	
	By Lemma \ref{lem:sameendpt}, for any $(i_0,...,i_m)\in\Sigma$ the vertices $t_{(i_0,...,i_m)}(\gamma_j)$ are the same for all $1\leq j\leq k$. But Lemma \ref{lem:addj} tells us that
	$$t_{(i_0,...,i_m,1,...,j-1)}(\gamma_j)=t_{(i_0,...,i_m,1,...,j)}(\gamma_j)$$
	for all $1\leq j\leq k$, so we deduce that
	$$t_{(i_0,...,i_m)}(\gamma_1)=t_{(i_0,...,i_m,1,...,k)}(\gamma_1).$$
	Now consider the path $\gamma:=\delta_{(1,...,k)}(\gamma_1)$, which lies in $\cap_{i=1}^k Y_i$ by Lemma \ref{lem:inintersection}. Lemma \ref{lem:splitdelta} implies that
	$$\delta_{(i_0,...,i_m,1,...,k)}(\gamma_1)=\delta_{(i_0,...,i_m)}(\gamma),$$
	hence
	\begin{equation}\label{gammagamma1}
t_{(i_0,...,i_m)}(\gamma_1)=t_{(i_0,...,i_m)}(\gamma).
	\end{equation}
The paths $\gamma_1,\gamma$ are both contained in $Y_1$, so $\gamma_1=\delta_{(1)}(\gamma_1)$ and $\gamma=\delta_{(1)}(\gamma)$, and (\ref{gammagamma1}) implies that $\gamma_1,\gamma$ have the same terminal vertex.
	Since (\ref{gammagamma1}) holds for all $(i_0,...,i_m)\in\Sigma$, we can apply Lemma \ref{lem:sameendpt} to deduce that the lift $\dot{\gamma}$ of $\gamma$ to $\dot{X}$ based at $\dot{x}$ has the same terminal vertex as $\dot{\gamma}_1$, namely the vertex $\dot{x}'$. So $\dot{\gamma}$ is the required path in $\cap_{i=1}^k\dot{Y}_i$ joining $\dot{x}$ and $\dot{x}'$.
\end{proof}

\begin{lem}
	The covering map $\dot{X}\to X$ is injective on $\cap_{i=1}^n \dot{Y}_i$.  
\end{lem}
\begin{proof}
Let $Z$ be the component of $\cap_{i=1}^n Y_i$ containing $x$.
The restriction of $\dot{X}\to X$ to $\cap_{i=1}^n \dot{Y}_i$ is a covering of $Z$, so our task is to show that this covering has degree 1.
We will do this by showing that every loop in $Z$ based at $x$ lifts to a loop in $\dot{X}$ based at $\dot{x}$.
If the walker traverses a loop $\gamma$ in $Z\subset\cap_{i=1}^n Y_i$ based at $x$ then all the imitators will also traverse $\gamma$.
Hence $\theta_x\cdot[\gamma]=\theta_x$ and $[\gamma]\in G_{\theta_x}$, so $\gamma$ lifts to a loop in $\dot{X}$ based at $\dot{x}$. 
\end{proof}

\bigskip
\section{Imitator covers}\label{sec:imitatorcovers}

In this section we offer an alternative construction of the imitator covers defined in Construction \ref{cons:hom} by interpreting the vertices and edges of the covers directly in terms of the movements of walker and imitator.
This alternative construction is formally the same as the canonical completion and retraction of Haglund--Wise \cite[Proposition 6.5]{HaglundWise08}, but is still described using the cartoon picture of walker and imitator from Construction \ref{cons:imitator}, so it ties together the walker-imitator viewpoint with the work of Haglund--Wise. Note that the canonical completion is defined slightly differently elsewhere in the literature, such as in \cite{WiseRiches}, by using a certain fibre product over a Salvetti complex of a RAAG -- but the two constructions coincide for local isometries of directly special cube complexes.

\begin{cons}(Imitator covers)\\\label{cons:imitatorcover}
	Let $\phi:Y\to X$ be a local isometry of directly special cube complexes. We consider the imitator and walker wandering around the 1-skeleta of $Y$ and $X$ respectively, as in Construction \ref{cons:imitator}. Define a cube complex $\sfC(Y,X)$ as follows (same notation as \cite{HaglundWise08}). The 0-skeleton $\sfC(Y,X)^0:=Y^0\times X^0$ is the set of possible positions of imitator and walker on vertices. The 1-skeleton describes all possible movements of walker and imitator: given $(y,x)\in \sfC(Y,X)^0$ and $e\in\link(x)$, suppose the walker traverses $e$ from $x$ to $x'$ while the imitator moves from $y$ to $y'$ (either traversing an edge $f$ from $y$ to $y'$ or staying put at $y=y'$), then $\sfC(Y,X)$ has an edge $\mathsf{e}(e;y,x)$ joining $(y,x)$ to $(y',x')$. Running this in reverse, if the walker traverses $e$ from $x'$ to $x$ then the imitator would move from $y'$ to $y$, and this defines the edge $\mathsf{e}(e;y',x')$ of $\sfC(Y,X)$, but in this case we define $\mathsf{e}(e;y,x)=\mathsf{e}(e;y',x')$ (one could think of these as two orientations of the same edge, but as we only work with 1-cells in this paper we will just consider them to be the same edge).
	
	By projecting to the positions of imitator and walker, we have combinatorial maps $\sfC(Y,X)^1\to Y^1$ and $\sfC(Y,X)^1\to X^1$, and the map $\sfC(Y,X)^1\to X^1$ is clearly a covering. By Lemma \ref{lem:gammasquare}, the boundary of each square in $X$ lifts to a collection of 4-cycles in $\sfC(Y,X)^1$, so we can form $\sfC(Y,X)^2$ by attaching squares with boundary maps equal to these 4-cycles. This gives a covering $\sfC(Y,X)^2\to X^2$, which can then be completed to a covering $\sfC(Y,X)\to X$ by adding higher dimensional cubes. Moreover, this will be a finite-sheeted covering if $Y$ is finite.
	
	The map $\sfC(Y,X)^1\to Y^1$ extends to a cellular map $r:\sfC(Y,X)\to Y$, called the \emph{canonical retraction}, as we now describe. Suppose $C$ is an $n$-cube in $\sfC(Y,X)$ with corner $(y,x)$, and say the edges of $C$ incident to $(y,x)$ project to edges $e_1,...,e_n\in\link(x)$. Suppose $f_1,...,f_m\in\link(y)$ satisfy $\phi(f_i)\parallel e_i$, and suppose that no $f\in\link(y)$ has $\phi(f)\parallel e_i$ for $m<i\leq n$. Then by arguments similar to those in Lemma \ref{lem:gammasquare}, we know that $f_1,...,f_m$ form the corner of a cube $C_Y$ in $Y$, and the projection $\sfC(Y,X)^1\to Y$ maps $C^1\to C^1_Y$, sending $\mathsf{e}(e_i;y,x)\mapsto f_i$ for $1\leq i\leq m$ and $\mathsf{e}(e_i;y,x)\mapsto y$ for $m<i\leq n$. So this extends to a map $C\to C_Y$ that collapses the dimensions of $C$ corresponding to $m<i\leq n$. These maps then fit together to give a cellular map $r:\sfC(Y,X)\to Y$.
	
	The map $Y^0\xhookrightarrow{}\sfC(Y,X),y\mapsto (y,\phi(y))$ extends to a combinatorial embedding $j:Y\xhookrightarrow{}\sfC(Y,X)$ called the \emph{canonical completion of $\phi:Y\to X$}. Note that the composition $r\circ j$ is the identity on $Y$, and we also have a commutative diagram:
	\begin{equation}\label{completioncommute}
		\begin{tikzcd}[
			ar symbol/.style = {draw=none,"#1" description,sloped},
			isomorphic/.style = {ar symbol={\cong}},
			equals/.style = {ar symbol={=}},
			subset/.style = {ar symbol={\subset}}
			]
			Y\ar{dr}[swap]{\phi}\ar{r}{j}&\sfC(Y,X)\ar{d}\\
			&X
		\end{tikzcd}
	\end{equation}
	
	In general $\sfC(Y,X)$ will not be connected.
	Fixing a basepoint $x\in X^0$, each component $C$ of $\sfC(Y,X)$ will contain a vertex $(y,x)$, and the based cover $(C,(y,x))\to(X,x)$ will correspond to the subgroup $G_y<G:=\pi_1(X,x)$ from Construction \ref{cons:hom}. Furthermore, the restriction of the canonical retraction to this component $r:(C,(y,x))\to(Y,y)$ will induce the homomorphism $\rho_{x,y}:G_y\to\pi_1(Y,y)$.
	If $\phi:(Y,y)\to(X,x)$ is a based map then $C$ is precisely the imitator cover of $\phi$ from Construction \ref{cons:hom}. Moreover, $C$ will contain $j(Y)$ in this case, and the map $j:Y\to C$ is the same as $j:Y\to\dot{X}$ from Proposition \ref{prop:completionhyps}.
\end{cons}

This construction naturally generalises to multiple imitators as we indicate below.

\begin{cons}(Imitator covers with multiple imitators)\\\label{cons:imitatormultiple}
	Let $Y_1,...,Y_n\to X$ be local isometries of directly special cube complexes.
	We can consider a set $\calI$ of imitators wandering around the $Y_i$ while a single walker wanders around in $X$. All the imitators try to copy the walker in the same way as in Construction \ref{cons:imitator}. Alternatively, some of the imitators could copy one another according to a hierarchy, as in Construction \ref{cons:hierarchy}. As in Construction \ref{cons:imitatorcover}, the movements of imitators and walker define a cover $\sfC_\calI(\{Y_i\},X)$ of $X$, whose vertices represent all possible positions of imitators and walker on the vertices of $\sqcup_i Y_i$ and $X$ respectively. This will be a finite-sheeted covering if both $\calI$ and the $Y_i$ are finite. For each $I\in\calI$ we will also have a cellular map $r_I:\sfC_\calI(\{Y_i\},X)\to\sqcup_i Y_i$ by taking the position of the imitator $I$.
\end{cons}

\begin{figure}[H]
	\centering
	\includegraphics[width=\textwidth,clip=true, trim=0 130 0 80]{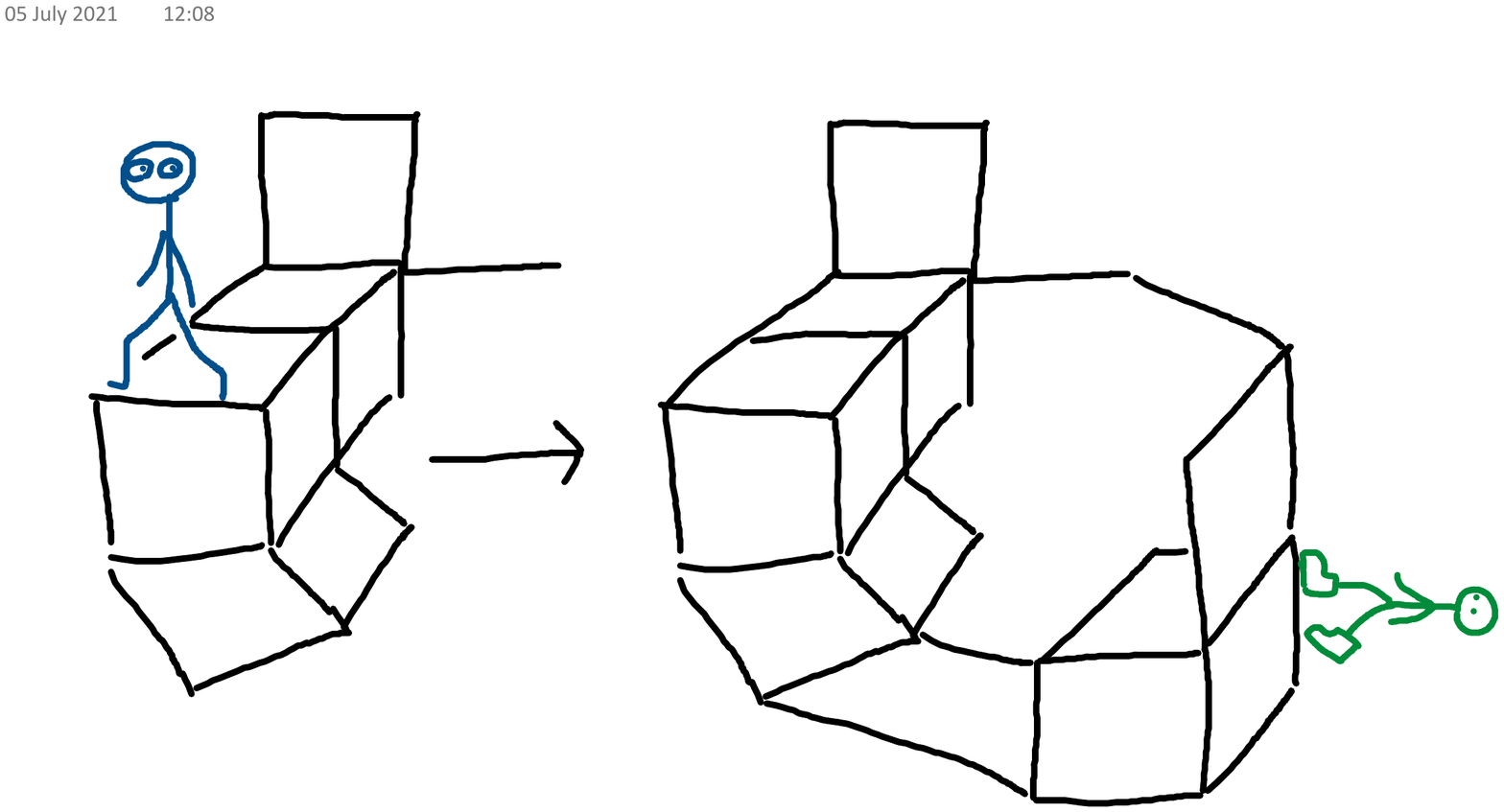}
	\caption{Cartoon of the walker and imitator.}
\end{figure}

\bibliographystyle{alpha}
\bibliography{Ref}

\end{document}